\documentclass[a4paper, 12pt]{amsart}
\usepackage[all]{xy}
\usepackage{hyperref}
\usepackage{amssymb, color, amsmath}
\usepackage{epsfig}
\usepackage{tikz}
\usepackage{mathtools}
\usepackage{enumerate}
\usepackage{graphicx}
\usepackage{tabularx}
\usepackage{lipsum}
\usepackage[normalem]{ulem}
\usetikzlibrary{patterns}
\usepackage{enumitem}
\usepackage[all]{xy}
\usepackage{color}
\usepackage{epsfig}

\hypersetup{colorlinks=true, linkcolor=blue, linkbordercolor=blue, citecolor=blue, citebordercolor=blue, linktocpage=true}

\definecolor{darkgreen}{rgb}{0.008,0.417,0.067}

\def\smallspace{\negthickspace}
\numberwithin{equation}{section}

\newcounter{theorem}
\numberwithin{theorem}{section}
\newtheorem{thm}[theorem]{Theorem}
\newtheorem{lemma}[theorem]{Lemma}
\newtheorem{prop}[theorem]{Proposition}
\newtheorem{cor}[theorem]{Corollary}

\newtheorem*{sas}{Standing Assumption}

\theoremstyle{remark}
\newtheorem*{remark*}{Remark}
\newtheorem{remark}[theorem]{Remark}

\newtheorem{example}[theorem]{Example}

\theoremstyle{definition}
\newtheorem{defn}[theorem]{Definition}

\newcommand{\espan}{\operatorname{span}}

\def\IC{\operatorname{IC}}

\newcommand{\CC}{\mathbb{C}}
\newcommand{\NN}{\mathbb{N}}
\newcommand{\TT}{\mathbb{T}}
\newcommand{\ZZ}{\mathbb{Z}}
\newcommand{\RR}{\mathbb{R}}

\newcommand{\Kk}{\mathcal{K}}

\newcommand{\Tt}{\mathcal{T}}

\newcommand{\Zz}{\mathcal{Z}}

\def\MCE{\operatorname{MCE}}

\newcommand\mydef{:=}

\date{\today}

\author{David Pask}
\email{dpask, asierako, asims@uow.edu.au}
\author{Adam Sierakowski}
\author{Aidan Sims}
\address{School of Mathematics and Applied Statistics \\
	University of Wollongong\\
	Wollongong NSW 2522\\
	AUSTRALIA}	

\thanks{This research was supported by Australian Research Council grant DP180100595.}

\title[Structure theory and stable rank]{Structure theory and stable rank for C*-algebras of finite higher-rank graphs}
\date{\today}

\begin{document}

\begin{abstract}
We study the structure and compute the stable rank of $C^*$-algebras of finite higher-rank
graphs. We completely determine the stable rank of the $C^*$-algebra when the $k$-graph either
contains no cycle with an entrance, or is cofinal. We also determine exactly which finite,
locally convex $k$-graphs yield unital stably finite $C^*$-algebras. We give several examples to
illustrate our results.
\end{abstract}

\maketitle
\section*{Introduction}
\setcounter{footnote}{1}

Stable rank is a non-commutative analogue of topological dimension for $C^*$-algebras introduced
by Rieffel in the early 1980s \cite{MR693043}, and widely used and studied ever since (see, for
example \cite{MR768308, MR1053808, MR1151561, MR1691013, MR1868695, MR2013159, MR2106263,
MR2266375, MR2476944, MR3471097, MR3725497, MR4165470}). The condition of having stable rank~1,
meaning that the invertible elements are dense in the $C^*$-algebra, has attracted significant
attention, in part due to its relevance to classification of $C^*$-algebras. Specifically, all
separable, simple, unital, nuclear, $\Zz$-stable $C^*$-algebras in the UCT class are classified by
their Elliott invariant \cite{Zstable, MR3546681, MR1403994, MR1745197, MR3418247}, and stable
rank distinguishes two key cases: the stably finite $C^*$-algebras in this class have stable
rank~1 \cite[Theorem~6.7]{MR2106263}, while the remainder are Kirchberg algebras with stable rank
infinity \cite[Proposition~6.5]{MR693043}. It follows that a simple $C^*$-algebra whose stable
rank is finite but not equal to~1 does not belong to the class of $C^*$-algebras classified by
their Elliott invariants.

Higher rank graphs (or $k$-graphs) $\Lambda$ are generalisations of directed graphs. They give
rise to an important class of $C^*$-algebras $C^*(\Lambda)$ due to their simultaneous concreteness
of presentation and diversity of structure \cite{MR2511133, MR2520478, MR2258220}. They provide
good test cases for general theory \cite{MR3150171, MR3263040} and have found unexpected
applications in general $C^*$-algebra theory. For example, the first proof that Kirchberg algebras
in the UCT class have nuclear dimension~1 proceeded by realising them as direct limits of 2-graph
$C^*$-algebras \cite{MR3345177}. Nevertheless, and despite their deceptively elementary
presentation in terms of generators and relations, $k$-graph $C^*$-algebras in general remain
somewhat mysterious---for example it remains an unanswered question whether all simple $k$-graph
$C^*$-algebras are $\Zz$-stable and hence classifiable. This led us to investigate their stable
rank. In this paper we shed some light on how to compute the stable rank of  $k$-graph
$C^*$-algebras; though unfortunately, the simple $C^*$-algebras to which our results apply all
have stable rank either 1 or $\infty$, so we obtain no new information about $\Zz$-stability or
classifiability.

This paper focuses on unital $k$-graph $C^*$-algebras. For $k=1$, i.e., for directed graph
$C^*$-algebras (unital or not), a complete characterisation of stable rank has been obtained
\cite{MR2001940, MR1868695, MR2059803}. In this paper our main contribution is for $k\geq 2$, a
characterisation of stable rank for $C^*$-algebras associated to
\begin{enumerate}
\item \label{case.one} finite $k$-graphs that have no cycle with an entrance, and
\item \label{case.two} finite $k$-graphs that are cofinal.
\end{enumerate}

In the first case  \eqref{case.one} we prove that such $k$-graphs are precisely the ones for which
the associated $C^*$-algebra is stably finite. Partial results on how to characterise stably
finite $k$-graph $C^*$-algebras have appeared in the past, see \cite{MR3507995, MR3311883,
MR3354440}. It turns out that in the unital situation, all such $C^*$-algebras are direct sums of
matrix algebras over commutative tori of dimension at most $k$; the precise dimensions of the tori
is determined by the degrees of certain cycles (called initial cycles) in the $k$-graph. Their
$C^*$-algebraic structure is therefore independent of the factorisation property that determines
how the one-dimensional subgraphs of a $k$-graph fit together to give it its $k$-dimensional
nature.

We also settle the second case  \eqref{case.two} where the $k$-graphs are cofinal using our
characterisation of stable finiteness in combination with a technical argument on the von Neumann
equivalence of (direct sums of) vertex projections. We initially obtained this result for selected
$2$-graphs using Python.

We now give a brief outline of the paper; Figure~\ref{fbuha} may also help the reader to navigate.
In Section~\ref{sec.1} we introduce terminology, including the notion (and examples) of an initial
cycle. In Section~\ref{two} we consider the stably finite case. In Proposition~\ref{stably.finite}
we prove that stable finiteness of $C^*(\Lambda)$ is equivalent to the condition that no cycle in
the $k$-graph $\Lambda$ has an entrance. In Theorem~\ref{direct.sum.thm} we characterise the
structure of $C^*(\Lambda)$ in the stably finite case and compute the stable rank of such algebras
in Corollary~\ref{stab.fin}. In Section~\ref{three} we characterise which $k$-graphs yield
$C^*$-algebras with stable rank~1 (Theorem~\ref{sr1} and Corollary~\ref{sr1.cor}) and show how the
dimension of the tori that form the components of their spectra can be read off from (the skeleton
of) the $k$-graph, see Proposition~\ref{ell.mu}.

In Section~\ref{four} we look  at $k$-graphs which are cofinal. Firstly, in
Proposition~\ref{s.r.simple}, we study the  special case when $C^*(\Lambda)$ is simple. Then, in
Theorem~\ref{thm.cofinal}, we compute stable rank when $\Lambda$ is cofinal and contains a cycle
with an entrance (so $C^*(\Lambda)$ is not stably finite). In Section~\ref{five} we illustrate the
difficulty in the remaining case where $\Lambda$ is not cofinal and contains a cycle with an
entrance by considering 2-vertex 2-graphs with this property. We are able to compute the stable
rank exactly for all but three classes of examples, for which the best we can say is that the
stable rank is either 2 or 3.

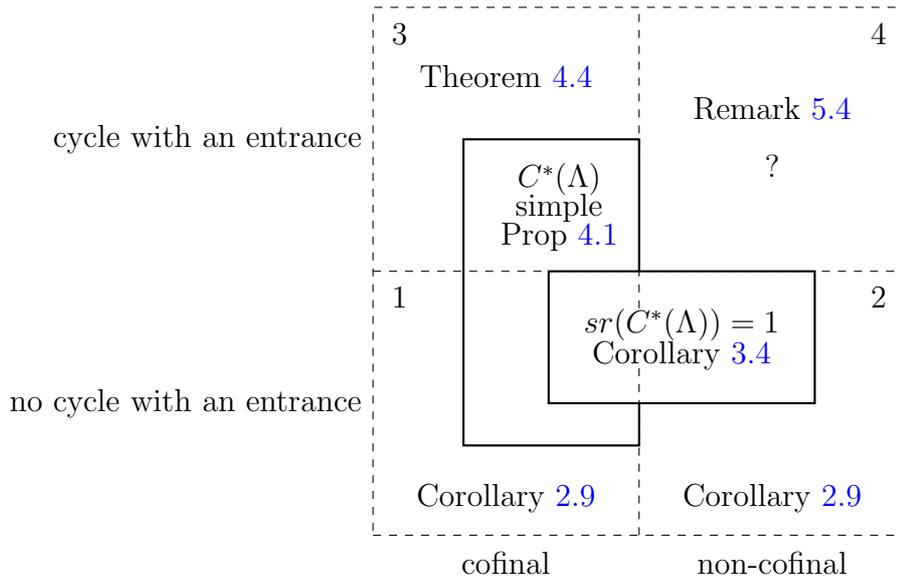
\begin{figure}
\centering
\[
\begin{tikzpicture}[scale=7]
\draw[dashed] (0,0)--(1,0)--(1,1)--(0,1)--(0,0);
\draw
  (0.05,0.45) node {1};
  \draw
    (0.05,0.95) node {3};
    \draw
      (0.95,0.95) node {4};
      \draw
        (0.95,0.45) node {2};
\draw[dashed] (0.5,0.43)--(0.5,1);
\draw[dashed] (0.5,0)--(0.5,.33);
\draw[dashed] (0,0.5)--(1,.5);
\draw[thick]
  (0.33,0.25) rectangle (0.83,.5);
 \draw[thick] (0.5, 0.25)--(0.5,0.17)--(0.17,0.17)--(0.17,0.75)--(0.5,0.75)--(.5,.5);
\draw
  (0,0.75) node[anchor=east]  {cycle with an entrance}
  (0,0.25) node[anchor=east]  {no cycle with an entrance}
  (0.25,-.01) node[anchor=north] {cofinal}
  (0.75,-.01) node[anchor=north] {non-cofinal};
\draw
  (0.35,0.675) node {$C^*(\Lambda)$}
    (0.35,0.62) node {simple}
        (0.35,0.565) node {Prop~\ref{s.r.simple}}
        (0.58,0.40) node  {$sr(C^*(\Lambda))=1$}
        (0.58,0.34) node  {Corollary~\ref{sr1.cor} }
    (0.25,0.07) node {Corollary~\ref{stab.fin}}
        (0.25,0.87) node {Theorem~\ref{thm.cofinal}}

  (0.75,0.07) node {Corollary~\ref{stab.fin}}
    (0.75,.7) node {?}
    (0.75,.81) node {Remark~\ref{unknown}};
\end{tikzpicture}
\]
\caption{Overview of some of our results. The ``?'' indicates unknown stable rank.} \label{fbuha}
\end{figure}

\section{Background}

\label{sec.1} In this section we recall the definition of stable rank, and the notions of stably
finite and purely infinite $C^*$-algebras. We also recall the definitions of $k$-graphs and their
associated $C^*$-algebras. We discuss the path space of a locally convex $k$-graph and describe
initial cycles and their periodicity group. The reader familiar with these terms can skim through
or skip this section.

\subsection{Stable rank of \texorpdfstring{$C^*$}{C*}-algebras}

\label{sec.1.3} Let $A$ be a unital $C^*$-algebra. Following \cite{MR2188261}, let
\[
    Lg_n(A):= \big\{(x_i)_{i=1}^n \in A^n : \exists (y_i)_{i=1}^n\in A^n \text{ such that } \sum_{i=1}^ny_ix_i = 1\big\}.
\]
The \emph{stable rank} of $A$, denoted $sr(A)$, is the smallest $n$ such that $Lg_n(A)$ is dense
in $A^n$, or $\infty$ if there is no such $n$. For a $C^*$-algebra $A$ without a unit, we define
its stable rank to be that of its minimal unitisation $\widetilde{A}$.

A $C^*$-algebra $A$ has stable rank one if and only if  the set $A^{-1}$ of invertible elements in
$A$ is dense in $A$. We will make frequent use of the following key results concerning the stable
rank of $C^*$-algebras of functions on tori, matrix algebras, stable $C^*$-algebras and direct
sums later in the paper:
\begin{enumerate}
\item\label{abelian.result}  $sr(C(\TT^{\ell}))=\left\lfloor {\ell/2}\right\rfloor+1$;
\item\label{matrix.result} $sr(M_n(A))=\left\lceil {(sr(A)-1)/n}\right\rceil+1$;
\item $sr(A\otimes \Kk) = 1$ if $sr(A) = 1$, and $sr(A \otimes \Kk) = 2$ if $sr(A) \not= 1$; and
\item\label{direct.sum} $sr(A\oplus B)=\max (sr(A),sr(B))$.
\end{enumerate}
Stable rank is in general not preserved under Morita equivalence (unless the stable rank is one).
For more details see \cite{MR693043}.

\subsection{Stably finite and purely infinite \texorpdfstring{$C^*$}{C*}-algebras}
A projection in a $C^*$-algebra is said to be \emph{infinite} if it is (von Neumann) equivalent to
a proper subprojection of itself. If a projection is not infinite it is said to be \emph{finite}.
A unital $C^*$-algebra $A$ is said to be \emph{finite} if its unit is a finite projection, and
\emph{stably finite} if $M_n(A)$ is finite for each positive integer $n$
\cite[Definition~5.1.1]{MR1783408}. We refer to \cite{MR3507995, MR3354440} for results about
stably finite graph $C^*$-algebras.

A simple $C^*$-algebra $A$ is \emph{purely infinite} if every nonzero hereditary sub-$C^*$-algebra
of $A$ contains an infinite projection. For the definition when $A$ is non-simple we refer the
reader to \cite{MR1759891}.

\subsection{Higher rank graphs}
Following \cite{MR1745529, MR2139097, MR1961175} we recall some terminology for $k$-graphs. For $k
\ge 1$, a \emph{$k$-graph} is a nonempty, countable, small category equipped with a functor $d :
\Lambda \to \NN^k$ satisfying the \emph{factorisation property}: For all $\lambda \in \Lambda$ and
$m,n\in \NN^k$ such that $d(\lambda) = m + n$ there exist unique $\mu, \nu\in \Lambda$ such that
$d(\mu) = m$, $d(\nu) = n$, and $\lambda = \mu\nu$. When $d(\lambda) = n$ we say $\lambda$ has
\emph{degree} $n$, and we define $\Lambda^n := d^{-1}(n)$. If $k=1$, then $\Lambda$ is isomorphic
to the free category generated by the directed graph with edges $\Lambda^1$ and vertices
$\Lambda^0$. The generators of $\NN^k$ are denoted $e_1, \dots ,e_k$, and $n_i$ denotes the
$i^{\textrm{th}}$ coordinate of $n \in \NN^k$. For $m, n \in \NN^k$, we write $m \le n$ if $m_i
\le n_i$ for all $i$, and we write $m\vee n$ for the coordinatewise maximum of $m$ and $n$, and $m
\wedge n$ for the coordinatewise minimum of $m$ and $n$.

If $\Lambda$ is a $k$-graph, its \emph{vertices} are the elements of $\Lambda^0$. The
factorisation property implies that the vertices are precisely the identity morphisms, and so can
be identified with the objects. For each $\lambda \in \Lambda$ the \emph{source} $s(\lambda)$ is
the domain of $\lambda$, and the \emph{range} $r(\lambda)$ is the codomain of $\lambda$ (strictly
speaking, $s(\lambda)$ and $r(\lambda)$ are the identity morphisms associated to the domain and
codomain of $\lambda$). Given $\lambda,\mu\in \Lambda$, $n \in \NN^k$ and $E\subseteq \Lambda$, we
define
\begin{align*}
\lambda E\mydef&\{\lambda\nu : \nu\in E, r(\nu)=s(\lambda)\},\\
E\mu\mydef&\{\nu\mu : \nu\in E, s(\nu)=r(\mu)\},\\
\Lambda^{\leq n} \mydef& \{\lambda \in \Lambda : d(\lambda) \leq n\text{ and }
    s(\lambda)\Lambda^{e_i} = \emptyset \text{ whenever } d(\lambda) + e_i \leq n \}.
\end{align*}

We say that a $k$-graph $\Lambda$ is \textit{row-finite} if $|v{\Lambda}^{n}|<\infty$ is finite
for each $n\in\mathbb{N}^k$ and $v\in{\Lambda}^0$, \emph{finite} if $|\Lambda^n|<\infty$ for all
$n\in \NN^k$, and \emph{locally convex} if for all distinct $i, j \in \{1, . . . , k\}$, and all
paths $\lambda \in \Lambda^{e_i}$ and $\mu \in \Lambda^{e_j}$ such that $r(\lambda)=r(\mu)$, the
sets $s(\lambda)\Lambda^{e_j}$ and $s(\mu)\Lambda^{e_i}$ are nonempty.

\begin{sas}
We have two standing assumptions. The first is that all of our $k$-graphs $\Lambda$ are finite.
This implies, in particular that they are row-finite. The second is that all of our $k$-graphs
$\Lambda$ are locally convex.
\end{sas}

A vertex $v$ is called a \emph{source} if there exist $i\leq k$ such that
$v\Lambda^{e_i}=\emptyset$. The term \emph{cycle}, distinct from ``generalised cycle''
\cite{MR2920846}, will refer to a path $\lambda \in \Lambda \setminus \Lambda^0$ such that
$r(\lambda) = s(\lambda)$.

We will occasionally illustrate $k$-graphs as $k$-coloured graphs. We refer to \cite{MR3056660}
for the details, but in short there is a one-to-one correspondence between $k$-graphs and
$k$-coloured graphs together with factorisation rules for bi-coloured paths of length~2 satisfying
an associativity condition \cite[Equation~(3.2)]{MR3056660}.

\subsection{Higher rank graph \texorpdfstring{$C^*$}{C*}-algebras}
Let $\Lambda$ be a row-finite, locally convex $k$-graph. Following \cite{MR1961175}, a
\emph{Cuntz--Krieger $\Lambda$-family} in a $C^*$-algebra $B$ is a function $s : \lambda \mapsto
s_\lambda$ from $\Lambda$ to $B$ such that
\begin{enumerate}\renewcommand{\theenumi}{CK\arabic{enumi}}
    \item\label{it:CK1} $\{s_v : v\in \Lambda^0\}$ is a collection of mutually orthogonal
        projections;
    \item\label{it:CK2} $s_\mu s_\nu = s_{\mu\nu}$ whenever $s(\mu) = r(\nu)$;
    \item\label{it:CK3} $s^*_\lambda s_\lambda = s_{s(\lambda)}$ for all $\lambda\in \Lambda$;
        and
    \item\label{it:CK4}  $s_v = \sum_{\lambda \in v\Lambda^{\leq n}} s_\lambda s_\lambda^*$ for
        all $v\in \Lambda^0$ and $n\in \NN^k$.
\end{enumerate}
The $C^*$-algebra $C^*(\Lambda)$ is the universal $C^*$-algebra generated by a Cuntz--Krieger
$\Lambda$-family. It is unital if and only if $|\Lambda^0|<\infty$, in which case $1=\sum_{v\in
\Lambda^0}s_v$. The universal family in $C^*(\Lambda)$ is denoted $\{s_\lambda :\lambda \in
\Lambda\}$.

\subsection{The path space of a \texorpdfstring{$k$}{k}-graph}
Let $\Lambda$ be a $k$-graph. For each path $\lambda\in \Lambda$, and $m \le n \le d(\lambda)$, we
denote by $\lambda(m,n)$ the unique element of $\Lambda^{n-m}$ such that $\lambda =
\lambda'\lambda(m,n)\lambda''$ for some $\lambda', \lambda''\in\Lambda$  with $d(\lambda') = m$
and $d(\lambda'') = d(\lambda) - n$. We abbreviate $\lambda(m,m)$ by $\lambda(m)$. A $k$-graph
morphism between two $k$-graphs is a degree preserving functor.

Following \cite{MR2920846}, for each $m \in (\NN \cup \{\infty\})^k$, we define a $k$-graph
$\Omega_{k,m}$ by $\Omega_{k,m} = \{(p,q) \in \NN^k \times \NN^k : p \le q \le m\}$ with range map
$r(p,q) = (p,p)$, source map $s(p,q) = (q,q)$, and degree map $d(p,q) = q-p$. We identify
$\Omega_{k,m}^0$ with $\{p \in \NN^k : p \le m\}$ via the map $(p,p) \mapsto p$. Given a $k$-graph
and $m \in \NN^k$ there is a bijection from $\Lambda^m$ to the set of morphisms $x : \Omega_{k,m}
\to \Lambda$, given by $\lambda \mapsto \big((p,q)\mapsto \lambda(p,q)\big)$; the inverse is the
map $x \mapsto x(0,m)$. Thus, for each $m\in \NN^k$ we may identify the collection of $k$-graph
morphisms from $\Omega_{k,m}$ to $\Lambda$ with $\Lambda^m$. We extend this idea beyond $m\in
\NN^k$ as follows: Given $m \in (\NN \cup \{\infty\})^k \setminus \NN^k$, we regard each $k$-graph
morphism $x : \Omega_{k,m} \to \Lambda$ as a path of degree $m$ in $\Lambda$ and write $d(x)\mydef
m$ and $r(x)\mydef x(0)$; we denote the set of all such paths by $\Lambda^m$. We denote by
$W_\Lambda$ the collection $\bigcup_{m \in (\NN \cup \{\infty\})^k} \Lambda^m$ of all paths in
$\Lambda$; our conventions allow us to regard $\Lambda$ as a subset of $W_\Lambda$. We call
$W_\Lambda$ the \emph{path space} of $\Lambda$.  We set
$$\Lambda^{\leq \infty}= \big\{ x\in W_\Lambda: x(n)\Lambda^{e_i}=\emptyset\ \textrm{whenever}\  n\leq d(x)\ \textrm{and} \  n_i = d(x)_i\big\},
$$
and for $v\in \Lambda^0$, we define $v\Lambda^{\leq \infty}:=\{x\in \Lambda^{\leq \infty}:
r(x)=v\}$. Given a cycle $\tau$, we define $\tau^\infty$ (informally written as
$\tau^\infty\mydef\tau\tau\tau\dots$) to be the unique element of $W_\Lambda$ such that
$d(\tau^\infty)_i$ is equal to $\infty$ when $d(\tau)_i > 0$ and equal to $0$ when $d(\tau)_i =
0$, and such that $(\tau^\infty)(n\cdot d(\tau), (n+1)\cdot d(\tau)) = \tau$ for all $n \in \NN$.

\subsection{Initial cycles and their periodicity group}
\label{init.cycles.intro} In this section we introduce initial cycles and their associated
periodicity group. As we will see in Corollary~\ref{stab.fin} and Theorem~\ref{sr1}, the
periodicity group plays an important role in the characterisation of stable rank.

Let $\lambda$ be a cycle in a row-finite, locally convex $k$-graph $\Lambda$ . We
say $\lambda$ is a \emph{cycle with an entrance} if there exists $\tau\in r(\lambda)\Lambda$ such
that $d(\tau)\leq d(\lambda^\infty)$ and $\tau\neq \lambda^\infty(0,d(\tau))$.

\begin{defn}{{(\cite{MR2920846})}}
\label{in.cycle} Let $\Lambda$ be a finite, locally convex $k$-graph that has no cycle with an
entrance. We call $\mu\in \Lambda$ an \emph{initial
cycle} if $r(\mu) = s(\mu)$ and if $r(\mu) \Lambda^{e_i} = \emptyset$ whenever $d(\mu)_i = 0$.
\end{defn}

\begin{remark}
While the wording of Definition~\ref{in.cycle} differs from that of \cite{MR2920846}, we will show
(in Proposition~\ref{stably.finite}) that for any $k$-graph $\Lambda$, a path $\mu\in \Lambda$ is
an initial cycle in the sense of Definition~\ref{in.cycle} if and only if it is an initial cycle
in the sense of \cite{MR2920846}.
\end{remark}

\begin{remark}
An initial cycle may be trivial, in the sense that it has degree 0, so it is in fact a vertex.
This vertex must then be a source, as for example $w_4$ in Figure~\ref{pic1.5}. It is not true
that every source is an initial cycle; for example $w_3$ in Figure~\ref{pic1.5} is a source but
not an initial cycle.
\end{remark}

As in \cite{MR2920846}, we let $\IC(\Lambda)$ denote the collection of all initial cycles in
$\Lambda$; if $\Lambda^0$ is finite and $\Lambda$ has no cycle with an entrance, then
$\IC(\Lambda)$ is nonempty---see Lemma~\ref{unit.sum}. A vertex $v\in \Lambda^0$ is said to be
\emph{on the initial cycle} $\mu$ if $v=\mu(p)$ for some $p\leq d(\mu)$\footnote{This is not the
definition in \cite[p.~202]{MR2920846}, but we expect this was the intended definition.}. We let
$(\mu^\infty)^0$ denote the collection of all vertices on an initial cycle $\mu$ and let $\sim$ be
the equivalence relation on $\IC(\Lambda)$ given by $\mu\sim \nu \Leftrightarrow
(\mu^\infty)^0=(\nu^\infty)^0$.

\begin{remark}
\label{initial.cycle.name} For a finite, locally convex $k$-graph that has no cycle with an
entrance, each initial cycle is an ``initial
segment'' in the following sense:
\begin{enumerate}
\item \label{gen.source} Every path with range on the initial cycle is in the initial cycle, so
    paths can not ``enter'' an initial cycle (see Lemma~\ref{lem.init2}).
\end{enumerate}

Without the assumption that $\Lambda^0$ is finite and $\Lambda$ has no cycle with
an entrance property \eqref{gen.source} might fail. This is for example the case for the
$1$-graph with one vertex and two edges representing the Cuntz algebra $\mathcal{O}_2$. This
suggests that, in general, a different terminology should perhaps be used.
\end{remark}

As in \cite{MR2920846} we associate a group $G_\mu$ to each initial cycle $\mu$. Let $\Lambda$ be
a finite, locally convex $k$-graph that has no cycle with an entrance. Let $\mu$ be an initial
cycle in $\Lambda$. If $\mu$ is not a vertex, we define
\begin{equation}
\label{groupofmu}
G_\mu\mydef\{m-n: n,m\leq d(\mu^{\infty}), \mu^{\infty}(m)=\mu^{\infty}(n)\},
\end{equation}
otherwise, we let $G_\mu\mydef\{0\}$.

\begin{defn}
\label{rank.Gmu} By \cite[Lemma~5.8]{MR2920846}, $G_\mu$ is a subgroup of $\ZZ^k$, and hence
isomorphic to $\ZZ^{\ell_\mu}$ for some $\ell_\mu\in \{0,\dots, k\}$: we often refer to $\ell_\mu$
as the \emph{rank} of $G_\mu$.
\end{defn}

\begin{remark}
It turns out that $\ell_\mu = |\{i\leq k: d(\mu)_i> 0\}|$---see Proposition~\ref{ell.mu}.
\end{remark}

\subsection{Examples of initial cycles}
\label{sec.ex}
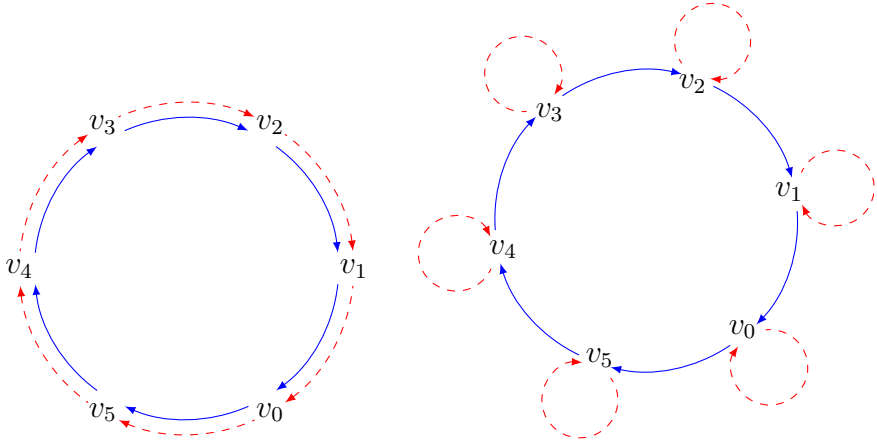
\begin{figure}
\begin{center}
\begin{tikzpicture}

\def \n {6}
\def \radiuss {2.0cm}
\def \radiusm {2.2cm}
\def \radiusb {2.2cm}
\def \margin {6}

\foreach \s in {5,...,0} {
  \node[circle, inner sep=0pt] at ({360/\n * (\s - 1)}:\radiusm) {$v_{\s}$};
  \draw[blue, <-, >=latex] ({360/\n * (\s - 1)+\margin}:\radiuss)
    arc ({360/\n * (\s - 1)+\margin}:{360/\n * (\s)-\margin}:\radiuss);

} \foreach \s in {5,...,0} {
  \draw[red, dashed, <-, >=latex] ({360/\n * (\s - 1)+\margin}:\radiusb)
    arc ({360/\n * (\s - 1)+\margin}:{360/\n * (\s)-\margin}:\radiusb);
}
\end{tikzpicture}
\ \ \
\begin{tikzpicture}

\def \n {6}
\def \radiuss {2.0cm}
\def \radiusm {2.2cm}
\def \radiusb {2.2cm}
\def \difff {2.15cm}
\def \margin {6}
\def \margins {150}

\foreach \s in {5,...,0} {
  \node[circle, inner sep=0pt] at ({360/(\n) * (\s - 0.81)}:\radiusm-8) {$v_{\s}$};
  \draw[blue, <-, >=latex] ({360/\n * (\s +\difff - 1)+\margin}:\radiuss)
    arc ({360/\n * (\s  +\difff - 1)+\margin}:{360/\n * (\s +\difff )-\margin}:\radiuss-1);

} \foreach \s in {5,...,0} {
  \draw[red, dashed, <-, >=latex] ({360/\n * (\s - 1)+\margin}:\radiusb-4)
    arc ({360/\n * (\s - 1)+\margin-310+\margins}:{360/\n * (\s - 1)+\margin+\margins}:0.5);
}
\end{tikzpicture}
\caption{Two $2$-graphs $\Lambda_1, \Lambda_2$, each containing lots of initial cycles, but only
one such up to $\sim$\,-equivalence.} \label{pic1}
\end{center}
\end{figure}
Figure~\ref{pic1} illustrates\footnote{We have illustrated  $\Lambda_1$ and $\Lambda_2$ as
2-coloured graphs, we refer to \cite{MR3056660} for details on how to visualise $k$-graphs as
colours graphs.} two examples of $2$-graphs $\Lambda_1$ and $\Lambda_2$ containing lots of initial
cycles. In fact, a cycle in either $\Lambda_1$ or $\Lambda_2$ is an initial cycle precisely if it
contains edges of both colours. Each initial cycle in either $\Lambda_i$ visits every vertex, so
any two initial cycles in $\Lambda_i$ are $\sim$\,-equivalent. A computation shows that for each
initial cycle $\mu$ in either $\Lambda_i$, $G_{\mu}\cong \ZZ^{2}$. Notice that each vertex on a
cycle in either $\Lambda_i$ has exactly one red (dashed) and one blue (solid) incoming and
outgoing edge and exactly one infinite path with range at that vertex.

Below we illustrate how the stable rank of each $C^*(\Lambda_i)$ can be computed. For this we need
some definitions and a lemma.
\begin{enumerate}
\item Fix an integer $n \ge 1$, let $L_n$ denote the connected $1$-graph with $n$ vertices $v_0
    , \ldots , v_{n-1}$ and $n$ morphisms $f_0 , \ldots , f_{n-1}$ of degree $1$ such that $s (
    f_i ) = v_{i+1\!\!\pmod{n}}$ and $r ( f_i ) = v_{i}$ for $0 \le i \le n-1$.
\item Let $( \Lambda_1 , d_1 )$ and $( \Lambda_2 , d_2 )$ be $k_1$-, $k_2$-graphs respectively,
    then $( \Lambda_1 \times \Lambda_2 , d_1 \times d_2 )$ is a  $(k_1 + k_2)$-graph where
    $\Lambda_1 \times \Lambda_2$ is the product category and $d_1 \times d_2 : \Lambda_1 \times
    \Lambda_2 \rightarrow \NN^{k_1 + k_2}$ is given by $d_1 \times d_2 ( \lambda_1 , \lambda_2 )
    = ( d_1 ( \lambda_1 ) , d_2 ( \lambda_2 ) ) \in \NN^{k_1} \times \NN^{k_2}$
    \cite[Proposition~1.8]{MR1745529}.
\item Let $f : \NN^\ell \rightarrow \NN^k$ be a monoid morphism. If $( \Lambda ,d)$ is a
    $k$-graph we may form the $\ell$-graph $f^* ( \Lambda )$ as follows: $f^* ( \Lambda ) = \{ (
    \lambda , n ) : d ( \lambda ) = f(n) \}$ with $d ( \lambda , n ) = n$, $s ( \lambda , n ) =
    s ( \lambda )$ and $r ( \lambda , n ) = r ( \lambda )$ \cite[Example~1.10]{MR1745529}.
\item Let $\Lambda$ be a $k$--graph and define $g_i : \NN \to \NN^k$ by $g_i (n) = n e_i$ for $1
    \leq i \leq k$ (so $\ell=1$). The $1$-graphs $\Lambda_i : = g_i^* ( \Lambda )$ are called
    the \emph{coordinate graphs} of $\Lambda$.
\end{enumerate}
\begin{lemma} \cite[Proposition 1.11, Corollary 3.5(iii), Corollary 3.5(iv)]{MR1745529}
\label{examples12}
\begin{enumerate}
\item Let $( \Lambda_i , d_i )$ be $k_i$-graphs for $i = 1, 2$, then $C^* ( \Lambda_1 \times
    \Lambda_2 ) \cong C^* ( \Lambda_1 ) \otimes C^* ( \Lambda_2 )$ via the map $s_{(\lambda_1 ,
    \lambda_2)} \mapsto s_{\lambda_1} \otimes s_{\lambda_2}$ for $( \lambda_1 , \lambda_2 ) \in
    \Lambda_1 \times \Lambda_2$.
\item Let $\Lambda$ be a $k$-graph and $f : \NN^\ell \rightarrow \NN^k $ a surjective monoid
    morphism. Then $C^* (f^* ( \Lambda ) ) \cong C^* ( \Lambda ) \otimes C ( \TT^{\ell - k})$.
\end{enumerate}
\end{lemma}

Let $\Lambda$ be a $1$-graph and define $f_1 : \NN^2 \rightarrow \NN$ by $( m_1, m_2) \mapsto m_1
+ m_2$. Then $f_1^* ( L_6 )$ is isomorphic to the $2$-graph $\Lambda$ shown on the left in
Figure~\ref{pic1}. The $2$-graph $L_6 \times L_1$ is isomorphic to the $2$-graph shown on the
right in Figure~\ref{pic1}. Using Lemma~\ref{examples12} and that $C^*(L_6)\cong M_6(C(\TT))$
\cite[Lemma~2.4]{MR1452183} we get $C^*(\Lambda_i)\cong C^* (L_6) \otimes C(\TT) \cong
M_6(C(\TT^2))$, $i=1,2$, so both have stable rank~2 as discussed in Section~\ref{sec.1.3}.

\section{Structure and stable rank in the stably finite case.}
\label{two} In this section we study finite $k$-graphs whose $C^*$-algebras are stably finite,
corresponding to boxes 1~and~2 in Figure~\ref{fbuha}. In Proposition~\ref{stably.finite} we show
that stable finiteness is equivalent to the lack of infinite projections and provide a
characterisation in terms of properties of the $k$-graph. We also provide a structure result and
compute stable rank of such $C^*$-algebras---see Theorems
\ref{direct.sum.thm}~and~\ref{stable.rank.thm}. We begin with four technical lemmas needed to
prove Theorem~\ref{direct.sum.thm}.

\begin{lemma}\label{lem.init2}
Let $\Lambda$ be a finite, locally convex $k$-graph that has no cycle with an
entrance. Let $v\in\Lambda^0$ be a vertex on an initial cycle $\mu\in \Lambda$. Then
\begin{enumerate}
\item \label{part1.init} there exist paths $\iota_v,\tau_v\in \Lambda$ such that
    $\mu=\iota_v\tau_v$ and $s(\iota_v)=v=r(\tau_v)$;
\item \label{part2.init} the path $\mu_v\mydef \tau_v\iota_v$ satisfies $r(\mu_v)= s(\mu_v)$,
    and $r(\mu_v) \Lambda^{e_i} = \emptyset$ whenever $d(\mu_v)_i = 0$;
\item \label{part3.init} if $f\in \Lambda^{e_i}$ is an edge with range $v$ on $\mu$, then
    $f=\mu_{v}(0, e_i)$ and $ \mu=\nu'f\nu''$ for some $\nu',\nu''\in \Lambda$;
\item \label{part4.init} if $n\leq d(\mu)$ and $\lambda\in v\Lambda^{\leq n}$, then
    $\lambda=\mu_{v}(0, d(\lambda))$; and
\item \label{part5.init} $s_{\tau_v}s_{\tau_v}^*=s_v$ and $s_{\tau_v}^*s_{\tau_v}=s_{s(\mu)}$.
\end{enumerate}
\end{lemma}
\begin{proof}
\eqref{part1.init}. Since $v$ is a vertex on $\mu$, we have $v=\mu(p)$ for some $p\leq d(\mu)$.
Set $\iota_v\mydef \mu(0,p)$ and $\tau_v\mydef \mu(p,d(\mu))$. Then $\mu=\iota_v\tau_v$ and
$s(\iota_v)=v=r(\tau_v)$.

\eqref{part2.init}. By property \eqref{part1.init}, $s(\iota_v)=v=r(\tau_v)$, so the path
$\mu_v\mydef \tau_v\iota_v\in \Lambda$ satisfies $r(\mu_v)=s(\mu_v)$. Suppose $r(\mu_v)
\Lambda^{e_i}$ is nonempty, say $\alpha\in r(\mu_v) \Lambda^{e_i}$. Then
$(\iota_v\alpha)(0,e_i)\in r(\mu)\Lambda^{e_i}$. Since $\mu$ is an initial cycle, it follows that
$d(\mu)_i \neq 0$.

\eqref{part3.init}. Suppose $f\in \Lambda^{e_i}$ is an edge with range $v$ on $\mu$. Since
$r(f)=v=r(\mu_{v})$, we have $f \in r(\mu_{v})\Lambda^{e_i}$. Now property \eqref{part2.init}
ensures that $d(\mu_v)_i  \neq 0$. Since $d(\mu_v)_i > 0$, there exists a path $\lambda \in
s(f)\Lambda^{\leq d(\mu_v)-e_i}$. Hence $f\lambda\in \Lambda^{\leq e_i}\Lambda^{\leq
d(\mu_v)-e_i}=\Lambda^{\leq d(\mu_v)}$. Now using that $\mu_v\in v\Lambda$ is a cycle and that
$\Lambda$ has no cycle with an entrance, it follows that $v\Lambda^{\leq d(\mu_v)}= \{\mu_v\}$.
Hence $\mu_v=f\lambda$. Since $f\lambda=\mu_v=\tau_v\iota_v$, we get $f=\tau_v(0,e_i)$ if
$d(\tau_v)_i>0$ and  $f=\iota_v(0,e_i)$ if $d(\tau_v)_i=0$.

\eqref{part4.init}. Fix $n\leq d(\mu)$ and $\lambda\in v\Lambda^{\leq n}$. If $d(\lambda) = 0$
then $\lambda = v$ and the statement is trivial, so we may assume that $\lambda\not\in \Lambda^0$.
Write $d(\lambda)=e_{i_1}+\dots+e_{i_m}$ where $i_1\dots,i_m\in \{1,\dots,k\}$. By the
factorisation property $\lambda=\lambda_1\dots \lambda_m$ for some $\lambda_j \in
\Lambda^{e_{i_j}}$. Repeated applications of part~\eqref{part3.init} give $\lambda_j =
\mu_{r(\lambda_j)}(0, e_{i_j})$ for $j \le m$. Since $d(\lambda)\leq d(\mu)$, it follows that
$\lambda=\mu_{v}(0, d(\lambda))$.

\eqref{part5.init}. Since $s(\tau_v)=s(\mu)$ we have $s_{\tau_v}^*s_{\tau_v}=s_{s(\mu)}$. For
$n\mydef d(\tau_v)$ notice that $n\leq d(\mu)$. Fix $\lambda\in v\Lambda^{\leq n}$. Using
part~\eqref{part4.init} we  have $\lambda=\mu_{v}(0, d(\lambda))$. Since $\tau_v\in v\Lambda^{\leq
n}$ and since $d(\lambda)\leq n$ we have $\tau_v=\mu_{v}(0, n)=\lambda \mu_{v}(d(\lambda), n)$.
But both $\tau_v$ and $\lambda$ belong to $\Lambda^{\leq n}$, so $\tau_v=\lambda$. Consequently
$v\Lambda^{\leq n}=\{\tau_v\}$ and $s_{\tau_v}s_{\tau_v}^*=\sum_{\lambda\in v\Lambda^{\leq
n}}s_\lambda s_\lambda^*=s_{v}$.
\end{proof}

\begin{lemma}
\label{unit.sum} Let $\Lambda$ be a finite, locally convex $k$-graph that has no
cycle with an entrance. Let $N\mydef(|\Lambda^0|,\dots, |\Lambda^0|)\in \NN^k$. Then
\begin{enumerate}
\item \label{part1.get} $\sum_{\lambda\in \Lambda^{\leq N}} s_\lambda  s_\lambda^* =
    1_{C^*(\Lambda)}$; and
\item \label{part2.get} for every $\lambda\in \Lambda^{\leq N}$, $s(\lambda)$ is a vertex on an
    initial cycle.
\end{enumerate}
\end{lemma}

\begin{proof}
For part~\eqref{part1.get} use $1=\sum_{v\in \Lambda^0}s_v=\sum_{v\in \Lambda^0}\sum_{\lambda\in
v\Lambda^{\leq N}}s_\lambda  s_\lambda^*=\sum_{\lambda\in \Lambda^{\leq N}}s_\lambda
s_\lambda^*.$ For part~\eqref{part2.get} we refer to the second paragraph of the proof of
\cite[Proposition 5.9]{MR2920846}.
\end{proof}

Recall that $(\mu^\infty)^0$ denotes the collection of all vertices on an initial cycle $\mu$. For
the terminology $\tau_{v}, v\in \Lambda^0$ in the following lemma see Lemma~\ref{lem.init2}.
\begin{lemma}
\label{matrix.units}
Let $\Lambda$ be a finite, locally convex $k$-graph that has no cycle with an
entrance. Fix an initial cycle $\mu\in \Lambda$.
Let $N\mydef(|\Lambda^0|,\dots, |\Lambda^0|)\in \NN^k$ and for each $\lambda,\nu\in \Lambda^{\leq
N}(\mu^\infty)^0$ set (using Lemma~\ref{lem.init2})
\begin{align*}
\theta_{\lambda,\nu}\mydef s_{\lambda\tau_{s(\lambda)}}s_{\nu\tau_{s(\nu)}}^*.
\end{align*}
Then the $\theta_{\lambda,\nu}$ are matrix units, i.e.,
$\theta_{\lambda,\nu}^*=\theta_{\nu,\lambda}$ and $\theta_{\lambda,\nu}\theta_{\gamma,
\eta}=\delta_{\nu,\gamma}\theta_{\lambda, \eta}$ in $C^*(\Lambda)$.
\end{lemma}

\begin{proof}
Firstly, note that each $\theta_{\lambda,\nu}$ makes sense because the source of $\lambda \in
\Lambda^{\leq N}(\mu^\infty)^0$ is a vertex on $\mu$ and $\tau_{s(\lambda)}$ is a path on the
initial cycle $\mu$ with range $s(\lambda)$.

Fix $\lambda, \nu,\gamma, \eta \in \Lambda^{\leq N}(\mu^\infty)^0$. Clearly
$\theta_{\lambda,\nu}^*=\theta_{\nu,\lambda}$. We claim that $\theta_{\lambda,\nu}\theta_{\gamma,
\eta}=\delta_{\nu,\gamma}\theta_{\lambda, \eta}$. To see this let $v\mydef r(\mu)$. Then
\begin{align}
\label{star.id}
s_{\nu\tau_{s(\nu)}}^*s_{\gamma\tau_{s(\gamma)}}=s_{\tau_{s(\nu)}}^*s_\nu^* s_\gamma s_{\tau_{s(\gamma)}}
\end{align}
is nonzero only if $r(\nu)=r(\gamma)$. Since $s_{r(\nu)} = \sum_{\alpha \in {r(\nu)}\Lambda^{\leq
N}} s_\alpha s_\alpha^*$, we have $(s_\nu s_\nu^*)(s_\gamma s_\gamma^*)=\delta_{\nu,\gamma}s_\nu
s_\nu^*$, so if \eqref{star.id} is nonzero, then $\nu=\gamma$ and then
\begin{align*}
s_{\nu\tau_{s(\nu)}}^*s_{\gamma\tau_{s(\gamma)}}=\delta_{\nu,\gamma}s_{\nu\tau_{s(\nu)}}^*s_{\gamma\tau_{s(\gamma)}}=\delta_{\nu,\gamma} s_{s(\nu\tau_{s(\nu)})}=\delta_{\nu,\gamma}s_{s(\tau_{s(\nu)})}=\delta_{\nu,\gamma}s_{v}.
\end{align*}
Hence $ \theta_{\lambda,\nu}\theta_{\gamma,
\eta}=s_{\lambda\tau_{s(\lambda)}}(\delta_{\nu,\gamma}s_{v})s_{\eta\tau_{s(\eta)}}^*=\delta_{\nu,\gamma}\theta_{\lambda,
\eta} $ as claimed.
\end{proof}
The next lemma is of general nature. See \cite[Lemma~3.3]{MR2520478} for a special case of this
result.

\begin{lemma}
\label{matrix.system} Suppose that $\{e_{ij}^{(k)} : 1\leq k\leq r, 1\leq i,j\leq n_k\}$ is a
system of matrix units in a unital $C^*$-algebra $A$, in the sense that
\begin{enumerate}
\item $e_{ij}^{(k)}e_{jl}^{(k)}=e_{il}^{(k)}$;
\item $e_{ij}^{(k)}e_{mn}^{(l)}=0$ if $k\neq l$ or if $j\neq m$;
\item $(e_{ij}^{(k)})^*=e_{ji}^{(k)}$; and
\item $\sum_{k=1}^r \sum_{i=1}^{n_k} e_{ii}^{(k)}=1$.
\end{enumerate}
For $k \leq r$ let $p^{(k)}\mydef \sum_{i=1}^{n_k} e_{ii}^{(k)}$. Suppose that for each $a\in A$,
$a=\sum_{k=1}^r p^{(k)}ap^{(k)}$. Then, for $1\leq k \leq r$, each $e_{11}^{(k)}$ is a projection
and $A\cong \bigoplus_{k=1}^{r} M_{n_k}(e_{11}^{(k)}Ae_{11}^{(k)})$.
\end{lemma}
\begin{proof}
Clearly $A\cong \bigoplus_{k=1}^{r}p^{(k)}Ap^{(k)}$ via $a\mapsto (p^{(1)}ap^{(1)},\dots,
p^{(r)}ap^{(r)})$ and inverse $(a^{(1)}, \dots, a^{(r)})\mapsto \sum_{k=1}^r a^{(k)}$. Routine
calculations show that for each $k\in \{1,\dots,r\}$, the elements $v_i\mydef e_{i1}^{(k)}$, $i=1,
\dots, n_k$, satisfy
$$v_i^*v_j=\delta_{i,j}e_{11}^{(k)} \quad \text{for} \quad 1\leq i,j\leq n_k, \quad \text{and}\quad p^{(k)}=\sum_{i=1}^{n_k}v_iv_i^*.$$
By \cite[Lemma~3.3]{MR2520478}, $p^{(k)}Ap^{(k)}\cong M_{n_k}(e_{11}^{(k)}Ae_{11}^{(k)})$
completing the proof.
\end{proof}

We now characterise the structure of $k$-graph $C^*$-algebras $C^*(\Lambda)$ such that $\Lambda$
is finite and has no cycle with an entrance. For the notions of $\IC(\Lambda)$, $\sim$,
$(\mu^\infty)^0$, and $\ell_\mu$ see Section~\ref{init.cycles.intro}. We note that in
Theorem~\ref{direct.sum.thm}, $\IC(\Lambda)\neq \emptyset$ and $\ell_\mu = |\{i\leq k: d(\mu)_i>
0\}|$ (see Lemma~\ref{unit.sum}\eqref{part2.get} and Proposition~\ref{ell.mu}).

\begin{thm}{(Structure theorem)}
\label{direct.sum.thm} Let $\Lambda$ be a finite, locally convex $k$-graph that has no cycle with an
entrance. For $N\mydef(|\Lambda^0|,\dots,
|\Lambda^0|)\in \NN^k$,
$$C^*(\Lambda)\cong \bigoplus_{[\mu] \in \IC(\Lambda)/\sim} M_{\Lambda^{\leq N}(\mu^\infty)^0}(s_{r(\mu)}C^*(\Lambda)s_{r(\mu)}),$$
and each $s_{r(\mu)}C^*(\Lambda)s_{r(\mu)}\cong C(\TT^{\ell_\mu})$.
\end{thm}

\begin{proof}
Evidently $\IC(\Lambda)/\smallspace\sim$ is finite since $\Lambda^0$ is finite. Let $I$ be a
maximal collection of initial cycles satisfying $(\mu^\infty)^0\cap (\nu^\infty)^0=\emptyset$ for
any $\mu\neq \nu\in I$. For each initial cycle $\mu\in I$ let $\{\theta_{\lambda,\nu}^{(\mu)}\}$
be the matrix units of Lemma~\ref{matrix.units}. We first prove that
$\{\theta_{\lambda,\nu}^{(\mu)}: \mu\in I \}$ is a system of matrix units, i.e.,
\begin{enumerate}
\item \label{m.one}
    $\theta_{\lambda\lambda'}^{(\mu)}\theta_{\lambda'\lambda''}^{(\mu)}=\theta_{\lambda\lambda''}^{(\mu)}$;
\item  \label{m.two} $\theta_{\lambda\lambda'}^{(\mu)}\theta_{\eta\eta'}^{(\nu)}=0$ if $\mu\neq
    \nu$ or if $\lambda'\neq \eta$;
\item  \label{m.three}  $(\theta_{\lambda\lambda'}^{(\mu)})^*=\theta_{\lambda'\lambda}^{(\mu)}$;
    and
\item \label{m.four}  $\sum_{\mu\in I}\sum_{\lambda \in \Lambda^{\leq N}(\mu^\infty)^0}
    \theta_{\lambda\lambda}^{(\mu)}=1$.
\end{enumerate}
We start with property \eqref{m.four}. Fix $\mu\in I$ and $\lambda\in \Lambda^{\leq
N}(\mu^\infty)^0$. Using Lemma~\ref{lem.init2}\eqref{part5.init} we have
$s_{\tau_{s(\lambda)}}s_{\tau_{s(\lambda)}}^*=s_{s(\lambda)}$. Hence
$\theta_{\lambda,\lambda}^{(\mu)}=
s_{\lambda\tau_{s(\lambda)}}s_{\lambda\tau_{s(\lambda)}}^*=s_\lambda s_\lambda^*$.
Lemma~\ref{unit.sum}\eqref{part1.get} now gives
\begin{equation}
\label{star.id2}
1=\sum_{\lambda\in \Lambda^{\leq N}}s_\lambda  s_\lambda^*=\sum_{\mu\in I} \sum_{\lambda\in \Lambda^{\leq N}(\mu^\infty)^0}s_\lambda  s_\lambda^*=\sum_{\mu\in I}\sum_{\lambda \in \Lambda^{\leq N}(\mu^\infty)^0} \theta_{\lambda\lambda}^{(\mu)}.
\end{equation}
Properties \eqref{m.one}--\eqref{m.three} follow from Lemma~\ref{unit.sum} and that
$\theta_{\lambda\lambda}^{(\mu)}\theta_{\lambda\lambda}^{(\nu)}=0$  whenever $\mu\neq \nu$ in $I$
(the latter is a consequence of property \eqref{m.four}).

For each $\mu \in I$, define $p^{(\mu)}\mydef \sum_{\lambda \in \Lambda^{\leq N}(\mu^\infty)^0}
\theta_{\lambda\lambda}^{(\mu)}$. Then~\eqref{star.id2} gives $\sum_{\mu\in I} p^{(\mu)}=1$. We
claim that
\[\textstyle
    A\mydef \big\{a\in C^*(\Lambda) : a=\sum_{\mu\in I} p^{(\mu)}ap^{(\mu)}\big\}
\]
is all of $C^*(\Lambda)$. Clearly $A$ is a closed linear subspace of $C^*(\Lambda)$. Fix $\alpha,
\beta\in \Lambda^{\leq N}$ such that $s(\alpha)=s(\beta)$. Using
Lemma~\ref{unit.sum}\eqref{part2.get} it follows that $s(\alpha)\in (\mu^\infty)^0$ for some
$\mu\in I$. Since $s_\alpha s_\alpha^*\leq p^{(\mu)}$ we get
\begin{equation*}
s_\alpha =p^{(\mu)}s_\alpha s_\alpha^* s_\alpha =p^{(\mu)}s_\alpha, \quad \text{and}\quad s_\alpha s_\beta^*=p^{(\mu)}s_\alpha s_\beta^*p^{(\mu)},
\end{equation*}
so $s_\alpha s_\beta^*\in A$. Using that $\espan\{s_\alpha s_\beta^*: \alpha, \beta\in
\Lambda^{\leq N}, s(\alpha)=s(\beta)\}$ is dense in  $C^*(\Lambda)$ we get $A=C^*(\Lambda)$ as
claimed.

For each $\mu \in I$, Lemma~\ref{lem.init2}\eqref{part3.init}--\eqref{part4.init} implies that
$r(\mu)\Lambda^{\leq N}$ contains exactly one path which we denote by $\lambda_\mu$. As in the
proof of  \eqref{m.four}, we have $\theta_{\lambda_\mu,\lambda_\mu}^{(\mu)}=s_{\lambda_\mu}
s_{\lambda_\mu} ^*$, so $s_{r(\mu)}= \sum_{\lambda \in r(\mu)\Lambda^{\leq N}} s_\lambda
s_\lambda^*=\theta_{\lambda_\mu,\lambda_\mu}^{(\mu)}$. Identifying $I$ with
$\IC(\Lambda)/\smallspace\sim$ via the map $\mu\mapsto [\mu]$, Lemma~\ref{matrix.system} provides
an isomorphism
$$C^*(\Lambda)\cong \bigoplus_{\mu \in \IC(\Lambda)/\sim} M_{\Lambda^{\leq N}(\mu^\infty)^0}(s_{r(\mu)}C^*(\Lambda)s_{r(\mu)}).$$
To see that each $s_{r(\mu)}C^*(\Lambda)s_{r(\mu)}\cong C(\TT^{\ell_\mu})$, see the proof of
\cite[Proposition~5.9]{MR2920846}.
\end{proof}

The main result of this section is the characterisation of stable rank for $k$-graph
$C^*$-algebras $C^*(\Lambda)$ such that $\Lambda$ is finite and has no cycle with an entrance.
Recall the notion of the floor and ceiling functions: for $x \in \mathbb{R}$, we write $\lfloor
x\rfloor \mydef \max\{n \in \mathbb{Z} : n \le x\}$ and $\lceil x\rceil \mydef \min\{n \in
\mathbb{Z} : n \ge x\}$.

\begin{thm}
\label{stable.rank.thm} Let $\Lambda$ be a finite, locally convex $k$-graph that
has no cycle with an entrance. For $N\mydef(|\Lambda^0|,\dots, |\Lambda^0|)\in
\NN^k$,
$$sr(C^*(\Lambda))=\max_{[\mu] \in \IC(\Lambda)/\sim} \left\lceil \frac{\left\lfloor \frac{\ell_\mu}{2}\right\rfloor}{|\Lambda^{\leq N}(\mu^\infty)^0|}\right\rceil+1.$$
\end{thm}

\begin{proof}
By Theorem~\ref{direct.sum.thm} and property \eqref{direct.sum} from Section \ref{sec.1.3},
$$sr(C^*(\Lambda))= \max_{[\mu] \in \IC(\Lambda)/\sim} M_{\Lambda^{\leq N}(\mu^\infty)^0}(s_{r(\mu)}C^*(\Lambda)s_{r(\mu)}),$$ and each
$s_{r(\mu)}C^*(\Lambda)s_{r(\mu)}\cong C(\TT^{\ell_\mu})$. Now using property
\eqref{matrix.result} from Section \ref{sec.1.3}, we get
$$sr(C^*(\Lambda))= \max_{[\mu] \in \IC(\Lambda)/\sim} \left\lceil {\frac{sr(s_{r(\mu)}C^*(\Lambda)s_{r(\mu)})-1}{{\Lambda^{\leq N}(\mu^\infty)^0}}}\right\rceil+1.$$
Finally, property \eqref{abelian.result} from Section \ref{sec.1.3}, gives
$sr(s_{r(\mu)}C^*(\Lambda)s_{r(\mu)}) -1=\left\lfloor {\ell_\mu/2}\right\rfloor$ for each $[\mu]
\in \IC(\Lambda)/\smallspace\sim$, completing the proof.
\end{proof}

Our next Proposition~\ref{stably.finite} characterises stable finiteness of $C^*$-algebras of
finite, locally convex $k$-graphs. For other results on stable finiteness of $C^*$-algebras
associated to row-finite $k$-graphs with no sources, see \cite{MR3507995, MR2920846}. Note that
the $C^*$-algebras satisfying the hypotheses of Proposition~\ref{stably.finite} are exactly those
shown in Figure~\ref{fbuha} in boxes 3~and~4.

We briefly introduce relevant terminology. Following \cite{MR2153156} we write $\MCE(\mu,
\nu)\mydef \mu\Lambda \cap \nu\Lambda\cap \Lambda^{d(\mu)\vee d(\nu)}$ for the set of all minimal
common extensions of $\mu,\nu\in \Lambda$. The cycle $\lambda$ is a \emph{cycle with an entrance
in the sense of \cite[Definition~3.5]{MR2920846}} if there exists a path $\tau\in
r(\lambda)\Lambda$ such that $\MCE(\tau, \lambda)=\emptyset$\footnote{Formally, if $\lambda$ is a
cycle, then $(\lambda,r(\lambda))$ is a generalised cycle in the sense of
\cite[Definition~3.1]{MR2920846}, and an entrance to $(\lambda, r(\lambda))$ is a path $\tau \in
s(r(\lambda))\Lambda$ such that $\MCE(r(\lambda)\tau, \lambda) = \emptyset$.}.
\begin{prop}
\label{stably.finite} Let $\Lambda$ be a finite, locally convex $k$-graph. With notation as above,
the following are equivalent:
\begin{enumerate}
\item \label{inf.1} $\Lambda$ has a cycle $\mu$ with an entrance;
\item \label{inf.2} $\Lambda$ has a cycle $\mu$ with an entrance in the sense of
    \cite[Definition~3.5]{MR2920846};
\item \label{inf.4} $C^*(\Lambda)$ contains an infinite projection; and
\item \label{inf.5} $C^*(\Lambda)$ is not stably finite.
\end{enumerate}
\end{prop}
\begin{proof}
To prove \eqref{inf.1}$\Rightarrow$\eqref{inf.2} let $\mu$ be a cycle with an entrance $\tau$, so
$\tau\in r(\mu)\Lambda$ satisfies $d(\tau)\leq d(\mu^\infty)$ and $\tau\neq
\mu^\infty(0,d(\tau))$. Fix $n\geq 1$ such that $nd(\mu)\geq d(\tau)$. Clearly $\tau\in
r(\mu^n)\Lambda$.  Then $\tau\neq \mu^\infty(0,d(\tau))=\mu^n(0,d(\tau))$, so $\MCE(\mu^n, \tau)=
(\mu^n\Lambda \cap \Lambda^{d(\mu^n)\vee d(\tau)})\cap \tau\Lambda\subseteq \{\mu^n\} \cap
\tau\Lambda=\emptyset$. For the proof of \eqref{inf.2}$\Rightarrow$\eqref{inf.4}, see
\cite[Corollary~3.8]{MR2920846}.

The implication \eqref{inf.4}$\Rightarrow$\eqref{inf.5} follows from
\cite[Lemma~5.1.2]{MR1783408}. It remains to prove \eqref{inf.5}$\Rightarrow$\eqref{inf.1}.  We
establish the contrapositive.  Suppose that condition \eqref{inf.1} does not hold, that is
$\Lambda$ has no cycle with an entrance. Theorem~\ref{direct.sum.thm} gives that
$C^*(\Lambda)$ is isomorphic to a direct sum of matrix algebras over commutative $C^*$-algebras,
hence  stably finite, so condition~\eqref{inf.5} does not hold.
\end{proof}

\begin{remark}
Our main results are for finite $k$-graphs so Proposition~\ref{stably.finite} is stated in that
context, but some of the implications hold more generally.
\begin{enumerate}
\item Only the proof of \mbox{\eqref{inf.5}\;${\implies}$\;\eqref{inf.1}} uses that $\Lambda$ is
    finite. The proofs of the implications \eqref{inf.1} ${\implies}$ \eqref{inf.2} ${\implies}$
    \eqref{inf.4} ${\implies}$ \eqref{inf.5} are valid for any locally convex row-finite
    $k$-graph.
\item Similarly, while \eqref{inf.1}~and~\eqref{inf.2} are not equivalent for arbitrary
    $k$-graphs, they \emph{are} equivalent for locally convex $k$-graphs, whether finite or not.
    To see this, suppose that $\lambda$ is a cycle in a locally convex $k$-graph and $\tau \in
    r(\lambda)\Lambda$ satisfies $\MCE(\lambda,\tau) = \emptyset$. Let $I = \{i \le k :
    d(\lambda)_i > 0\}$, let $m_I = \sum_{i \in I} d(\tau)_i e_i$ and $m' := d(\tau) - m_I$, and
    factorise $\tau = \tau_I \tau'$ with $d(\tau_I) = m_I$. If $\tau_I \not= (\lambda^\infty)(0,
    m_I)$, then $\lambda$ is a cycle with an entrance as required, so we may assume that $\tau_I
    = (\lambda^\infty)(0,m_I)$. So replacing $\lambda$ with $(\lambda^\infty)(m_I, m_I +
    d(\lambda))$ and $\tau$ with $\tau'$ we may assume that $d(\tau) \wedge d(\lambda) = 0$.
    Since $\Lambda$ is locally convex, a quick inductive argument shows that there exists $\mu
    \in s(\tau)\Lambda^{d(\lambda)} \not= \emptyset$. Factorise $\tau\mu = \alpha\beta$ with
    $d(\alpha) = d(\mu) = d(\lambda)$. Since $\MCE(\tau,\lambda) = \emptyset$, we must have
    $\alpha \not= \lambda$ and in particular $d(\alpha) = d(\lambda) < d(\lambda^\infty)$ and
    $\alpha \not= (\lambda^\infty)(0, d(\alpha))$. So once again $\lambda$ is a cycle with an
    entrance.
\end{enumerate}
\end{remark}

\begin{cor}
\label{stab.fin} Let $\Lambda$ be a finite, locally convex $k$-graph. Suppose that $\Lambda$ has
no cycle with an entrance (i.e., $C^*(\Lambda)$ is stably finite). For $N\mydef(|\Lambda^0|,\dots,
|\Lambda^0|)\in \NN^k$,
$$sr(C^*(\Lambda))=\max_{[\mu] \in \IC(\Lambda)/\sim} \left\lceil \frac{\left\lfloor \frac{\ell_\mu}{2}\right\rfloor}{|\Lambda^{\leq N}(\mu^\infty)^0|}\right\rceil+1.$$
\end{cor}

 \begin{remark}
A cycle with an incoming edge may fail to be a cycle with an entrance. This is for example the
case for any of the red (dashed) cycles in Figure~\ref{pic2}.
\end{remark}

\begin{example}
In this example we consider the $2$-graph $\Lambda$ in Figure~\ref{pic2}.
\begin{figure}
\begin{center}
\begin{tikzpicture}

\def \n {2}
\def \radiuss {1.0cm}
\def \radiusm {1.15cm}
\def \radiusb {1.15cm}
\def \margin {6}
  \node[circle, inner sep=0pt] (v0) at (-3,.5) {$e$};
  \node[circle, inner sep=0pt] (v1) at (-1,0) {$\ $};
  \node[circle, inner sep=0pt] (v2) at (1, 0) {$\ $};
  \node[inner sep=1pt, circle] (27) at (-2.5,0) {$v$};	
   \path[->,every loop/.style={looseness=14}, >=latex] (27)
			 edge  [in=150,out=210,loop, red, dashed, >=latex] ();	
\foreach \s in {0,1} {
  \node[circle, inner sep=0pt] (l\s) at ({360/\n * (\s - 1)}:\radiusm) {$v_{\s}$};
  \draw[blue, <-, >=latex] ({360/\n * (\s - 1)+\margin}:\radiuss)
    arc ({360/\n * (\s - 1)+\margin}:{360/\n * (\s)-\margin}:\radiuss);
  \draw[->, blue,  out=170, in=10,  >=latex] (l\s) to (27);

} \foreach \s in {0,1} {
  \draw[red, dashed, <-, >=latex] ({360/\n * (\s - 1)+\margin}:\radiusb)
    arc ({360/\n * (\s - 1)+\margin}:{360/\n * (\s)-\margin}:\radiusb);
}
\end{tikzpicture}
\caption{An example of a $2$-graph $\Lambda$ with $C^*(\Lambda)$ of stable rank 2.} \label{pic2}
\end{center}
\end{figure}
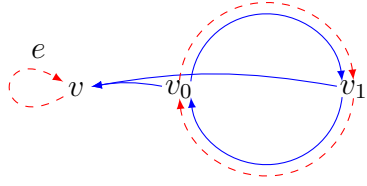
As before, we refer to \cite{MR3056660} for details on how to illustrate 2-graphs as a 2-coloured
graph. Here we use blue (solid) and red (dashed) as the first and second colour. We use our
results to compute the structure and stable rank of $C^*(\Lambda)$. Firstly notice that red the
cycle $\nu\in v\Lambda^{e_2}$ based on $v$ is not an initial cycle because, $d(\nu)_1 = 0$, but
$r(\nu) \Lambda^{e_1} \neq \emptyset$. However, the cycle $\mu\in v_0\Lambda^{e_1+e_2}$ is an
initial cycle. There are many other initial cycles, but they are all $\sim$-equivalent to $\mu$.
So $ \IC(\Lambda)/\smallspace\sim\ = \{[\mu]\}$. Since the vertices on a path $\lambda\in
v_0\Lambda$ alternate between $v_0$ and $v_1$ as we move along the path, it follows that
$\mu^{\infty}(m)=v_{m_1+m_2\pmod{2}}$ for each $m\in \NN^2$. Hence
\begin{align*}
G_\mu&=\{m-n: n,m\leq d(\mu^{\infty}), \mu^{\infty}(m)=\mu^{\infty}(n)\}\\
&=\{m-n: n,m\in \NN^2, v_{m_1+m_2\pmod{2}}=v_{n_1+n_2\pmod{2}}\}\\
&=\{(m_1-n_1, m_2-n_2): n_i,m_i\in \NN, m_1-n_1 = - (m_2-n_2)\pmod{2}\}\\
&=\{(k_1, k_2): k_i\in \ZZ, k_1 = -k_2\pmod{2}\}\\
&=\{k\in \ZZ^2: k_1 + k_2\ \text{is even}\}\\
&=\ZZ(1,1)+\ZZ(0,2),\\
&\cong \ZZ^2.
\end{align*}
We deduce that $\ell_\mu=\text{rank}(G_\mu)=2$. Now set $N=(|\Lambda^0|,|\Lambda^0|) =(3,3)$. As
mentioned, modulo $\sim$, there is only one initial cycle $\mu$, so any path in $\Lambda^{\leq N}$
has its source on $\mu$. Hence $\Lambda^{\leq N}(\mu^\infty)^0=\Lambda^{\leq N}=v\Lambda^{\leq
N}\sqcup v_0\Lambda^{\leq N}\sqcup v_1\Lambda^{\leq N}$. By Lemma~\ref{lem.init2},
$|v_0\Lambda^{\leq N}|=|v_1\Lambda^{\leq N}|=1$. Using the factorisation property to push the red
edges to the start of a path and uniqueness of such paths on $\mu$, we have $|v\Lambda^{\leq
N}|=|v\Lambda^{(3,3)}|=|v\Lambda^{(3,0)}|=|v\Lambda^{(1,0)}|=2$. Hence $C^*(\Lambda)\cong
M_{4}(C(\TT^{2}))$ and $sr(C^*(\Lambda))=\left\lceil \left\lfloor \frac{2}{2}\right\rfloor /
4\right\rceil+1=2$ by Theorem~\ref{stable.rank.thm}.
\end{example}

\begin{example}
In the following let $\Lambda_1, \Lambda_2$ and $\Lambda_3, \Lambda_4, \Lambda_5$ be the
$2$-graphs in Figure~\ref{pic1} and Figure~\ref{pic3} respectively. Up to a swap of the colours
these five examples make up all the examples of $2$-graphs on 6 vertices with only one initial
cycle up to $\sim$\,-equivalence and with all vertices on that initial cycle.
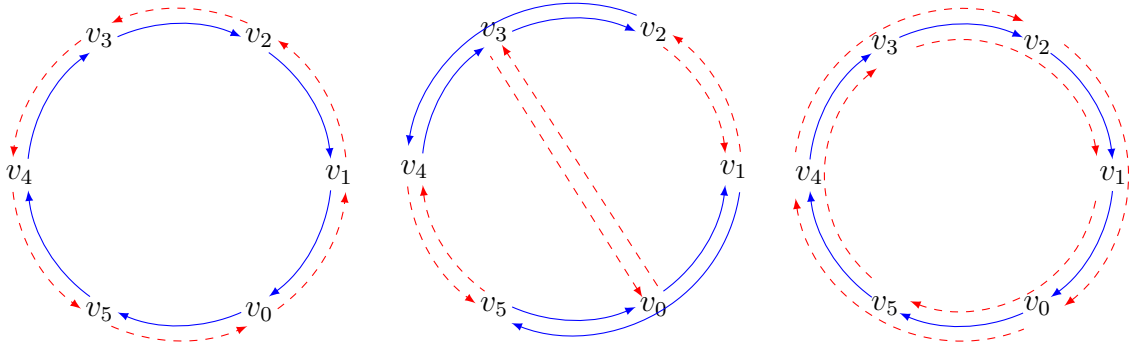
\begin{figure}
\begin{center}
\begin{tikzpicture}
\def \n {6}
\def \radiuss {2.0cm}
\def \radiusm {2.1cm}
\def \radiusb {2.2cm}
\def \margin {6}
\foreach \s in {5,...,0} {
  \node[circle, inner sep=0pt] at ({360/\n * (\s - 1)}:\radiusm) {$v_{\s}$};
  \draw[blue, <-, >=latex] ({360/\n * (\s - 1)+\margin}:\radiuss)
    arc ({360/\n * (\s - 1)+\margin}:{360/\n * (\s)-\margin}:\radiuss);

} \foreach \s in {5,...,0} {
  \draw[red, dashed, ->, >=latex] ({360/\n * (\s - 1)+\margin}:\radiusb)
    arc ({360/\n * (\s - 1)+\margin}:{360/\n * (\s)-\margin}:\radiusb);
}
\end{tikzpicture}
\ \ \
\begin{tikzpicture}
\def \n {6}
\def \radiuss {2.0cm}
\def \radiusm {2.1cm}
\def \radiusb {2.2cm}
\def \margin {6}
\draw[red, dashed, ->, >=latex] (1.1,-1.55) -- (-0.95,1.65);
\draw[red, dashed, <-, >=latex] (0.9,-1.7) -- (-1.1,1.5);

\foreach \s in {5,...,0} {
  \node[circle, inner sep=0pt] at ({360/\n * (\s - 1)}:\radiusm) {$v_{\s}$};
} \foreach \s in {2,3} {
  \draw[blue, <-, >=latex] ({360/\n * (\s - 1)+\margin}:\radiuss)
    arc ({360/\n * (\s - 1)+\margin}:{360/\n * (\s)-\margin}:\radiuss);
} \foreach \s in {0,5} {
  \draw[blue, ->, >=latex] ({360/\n * (\s - 1)+\margin}:\radiuss)
    arc ({360/\n * (\s - 1)+\margin}:{360/\n * (\s)-\margin}:\radiuss);
} \foreach \s in {1,4} {
  \draw[red, dashed, <-, >=latex] ({360/\n * (\s - 1)+\margin}:\radiuss)
    arc ({360/\n * (\s - 1)+\margin}:{360/\n * (\s)-\margin}:\radiuss);
} \foreach \s in {1,4} {
  \draw[red, dashed, ->, >=latex] ({360/\n * (\s - 1)+\margin}:\radiusb)
    arc ({360/\n * (\s - 1)+\margin}:{360/\n * (\s)-\margin}:\radiusb);
} \def \n {3}
\def \radius {1.8cm}
\def \radiusm {2.2cm}
\def \margin {8}
\foreach \s in {0}
{
  \node[circle, inner sep=0pt] at ({360/\n * (\s - 1)}:\radius) {$$};
  \draw[blue, <-, >=latex] ({360/\n * (\s - 1)+\margin}:\radiusm)
    arc ({360/\n * (\s - 1)+\margin}:{360/\n * (\s)-\margin}:\radiusm);
}

\foreach \s in {1}
{
  \node[circle, inner sep=0pt] at ({360/\n * (\s - .5)}:\radius) {$$};
  \draw[blue, ->, >=latex] ({360/\n * (\s - .5)+\margin}:\radiusm)
    arc ({360/\n * (\s - .5)+\margin}:{360/\n * (\s+.5)-\margin}:\radiusm);
}
\end{tikzpicture}
\ \ \
\begin{tikzpicture}

\def \n {6}
\def \radiuss {2cm}
\def \radiusm {2cm}
\def \radiusb {2.2cm}
\def \margin {6}

\foreach \s in {5,...,0} {
  \node[circle, inner sep=0pt] at ({360/\n * (\s - 1)}:\radiusm) {$v_{\s}$};
  \draw[blue, <-, >=latex] ({360/\n * (\s - 1)+\margin}:\radiuss)
    arc ({360/\n * (\s - 1)+\margin}:{360/\n * (\s)-\margin}:\radiuss);

} \foreach \s in {} {
  \draw[blue, dashed, <-, >=latex] ({360/\n * (\s - 1)+\margin}:\radiusb)
    arc ({360/\n * (\s - 1)+\margin}:{360/\n * (\s)-\margin}:\radiusb);
}
\def \n {3}
\def \radius {1.8cm}
\def \radiusm {2.2cm}
\def \margin {8}

\foreach \s in {1,...,\n}
{
  \node[circle, inner sep=0pt] at ({360/\n * (\s - 1)}:\radius) {$$};
  \draw[red, dashed, <-, >=latex] ({360/\n * (\s - 1)+\margin}:\radius)
    arc ({360/\n * (\s - 1)+\margin}:{360/\n * (\s)-\margin}:\radius);
}

\foreach \s in {1,...,\n}
{
  \node[circle, inner sep=0pt] at ({360/\n * (\s - .5)}:\radius) {$$};
  \draw[red, dashed, <-, >=latex] ({360/\n * (\s - .5)+\margin}:\radiusm)
    arc ({360/\n * (\s - .5)+\margin}:{360/\n * (\s+.5)-\margin}:\radiusm);
}

\end{tikzpicture}
\caption{Another three examples of $2$-graphs with lots of initial cycles, but only one such up to $\sim$\,-equivalence.} \label{pic3}
\end{center}
\end{figure}
Let $\mu_i$ denote such an initial cycle in $\Lambda_i$. For $0 \le j \le n-1$ define $f_j : \NN^2
\to \NN$ by $f_j ( m_1 , m_2 ) = m_1 + j m_2$.  With the terminology of Section~\ref{sec.ex} one
can show that $\Lambda_3=f_5^* ( L_6 )$, that $\Lambda_4=L_2 \times L_3$, and that
$\Lambda_5=f_2^* ( L_6 )$. Hence
\[
C^*(\Lambda_i)\cong M_{6}(C(\TT^{\ell_{\mu_i}})) \text{ and } sr(C^*(\Lambda))=\ell_{\mu_i}=2\ \ \  \text{for each}\ i=1,\dots,5.
\]
In particular $C^*(\Lambda_i)\cong C^*(\Lambda_j)$ for all $i,j$. These five examples indicate how
the number of colours and vertices impacts the structure of the corresponding $C^*$-algebras. The
next Proposition~\ref{init.prop} verifies this.
\end{example}

In the following, ``the number of vertices'' on an initial cycle $\mu$ means $|(\mu^\infty)^0|$,
and ``the number of colours'' means $|\{i\leq k: d(\mu)_i> 0\}|$. The proof borrows material from
an independent result (Proposition~\ref{ell.mu}).

\begin{prop}
\label{init.prop} Let $\Lambda$ be a finite, locally convex $k$-graph on $n=|\Lambda^0|$ vertices.
Suppose that $\Lambda$ has no cycle with an entrance and $\Lambda$ has exactly one initial cycle,
up to $\sim$\,-equivalence, with $n$ vertices and $\ell$ colours. Then
$$C^*(\Lambda)\cong M_n(C(\TT^\ell)), \text{ and } \ sr(C^*(\Lambda))=\lceil \lfloor \ell/2\rfloor/n\rceil+1.$$
\end{prop}

\begin{proof}
Let $\mu$ be an initial cycle in $\Lambda$. By Theorem~\ref{direct.sum.thm} and
Theorem~\ref{stable.rank.thm} we have
$$sr(C^*(\Lambda))=\left\lceil \frac{\left\lfloor \frac{\ell_\mu}{2}\right\rfloor}{|\Lambda^{\leq (n,\dots, n)}|}\right\rceil+1, \ \ \ C^*(\Lambda)\cong M_{\Lambda^{\leq (n,\dots, n)}}(C(\TT^{\ell_\mu})),$$
where $\ell_\mu$ is the rank of the periodicity group associated to $\mu$
(Definition~\ref{rank.Gmu}). By the factorisation property $\Lambda$ has no sources, so
$\Lambda^{\leq (n,\dots, n)}=\Lambda^{(n,\dots, n)}$. Lemma~\ref{lem.init2}\eqref{part3.init}
implies that $|\Lambda^{(n,\dots, n)}|=n$. By Proposition~\ref{ell.mu}, $\ell_\mu=\ell$. Combining
these results gives $C^*(\Lambda)\cong M_n(C(\TT^\ell))$, and $sr(C^*(\Lambda))=\lceil \lfloor
\ell/2\rfloor/n\rceil+1$.
\end{proof}

\begin{remark}
To keep the statement of Proposition~\ref{init.prop} short and clean we insisted that
$(\mu^\infty)^0=\Lambda^0$, but more general results can be obtained using
Theorem~\ref{direct.sum.thm} and Theorem~\ref{stable.rank.thm}.
\end{remark}

\section{Stable rank one.}
\label{three} In this section we characterise which finite $k$-graphs have $C^*$-algebras of
stable rank~1---see Theorem~\ref{sr1} and Corollary~\ref{sr1.cor}. We note that Theorem~\ref{sr1}
is in large contained in \cite{MR2920846} and we have structured the proof accordingly.

\begin{thm}
\label{sr1} Let $\Lambda$ be a finite, locally convex $k$-graph. Then $sr(C^*(\Lambda))=1$ if and
only if $C^*(\Lambda)$ is (stably) finite and $\max_{\mu\in IC(\Lambda)} \ell_{\mu}=1$.
\end{thm}
\begin{proof}
Suppose that $sr(C^*(\Lambda))=1$. Then $sr(M_n(C^*(\Lambda)))=1$  for each positive integer $n$
\cite[Theorem~6.1]{MR693043}. Hence each $M_n(C^*(\Lambda))$ is finite \cite[V.3.1.5]{MR2188261},
which implies that $C^*(\Lambda)$ is stably finite. Since $C^*(\Lambda)$ is finite it does not
contain any infinite projections \cite[Lemma~5.1.2]{MR1783408}.
Hence by \cite[Proposition~5.9]{MR2920846} there exist $n \ge 1$ and $l_1, \dots, l_n \in \{0,
\dots, k\}$ such that $C^*(\Lambda)$ is stably isomorphic to $\bigoplus^n_{i=1} C(\TT^{l_i})$.
Since $sr(C^*(\Lambda))=1$, we deduce that $sr(\bigoplus_{i=1}^n C(\TT^{l_i}))=1$, because stable
rank~1 for unital $C^*$-algebras is preserved my stable isomorphism \cite[Theorem~3.6]{MR693043}.
By property~\ref{direct.sum} in Section~\ref{sec.1.3}, we have $sr(\bigoplus_{i=1}^n
C(\TT^{l_i}))=\max_{i=1}^n C(\TT^{l_i})$. For each $i=1,\dots,n$ we use \cite[Proposition
1.7]{MR693043} to deduce that $sr(C(\TT^{l_i}))=\lfloor l_i/2\rfloor+1$, where
$\lfloor\cdot\rfloor$ denotes ``integer part of''. Hence $\max_{i=1}^n l_i=1$.

By inspection of the proof of \cite[Proposition~5.9]{MR2920846} it is clear that each of the
integers $l_i$ is the rank of $\mu$ for some $\mu\in IC(\Lambda)$, so $\max_{\mu\in IC(\Lambda)}
\ell_{\mu}\geq \max_{i=1}^n l_i=1$. For each $\mu,\nu\in IC(\Lambda)$ define $P_\nu\mydef
\sum_{v\in (\nu^\infty)^0}s_v$ and $\mu\sim \nu \Leftrightarrow (\mu^\infty)^0=(\nu^\infty)^0$.
Since $P_\nu=P_\mu$ whenever $\mu\sim \nu$, the proof of \cite[Proposition~5.9]{MR2920846} implies
that for each $\mu\in IC(\Lambda)$, we have $\ell_\mu=l_i$ for some $i\in \{1,\dots,n\}$.
Consequently, $\max_{\mu\in IC(\Lambda)} \ell_{\mu}= \max_{i=1}^n l_i$.

Conversely, suppose $C^*(\Lambda)$ is finite and $\max_{\mu\in IC(\Lambda)} \ell_{\mu}=1$. By
\cite[Lemma~5.1.2]{MR1783408}, $C^*(\Lambda)$ has no infinite projections. So
\cite[Corollary~5.7]{MR2920846} implies that $C^*(\Lambda)$ is stably isomorphic to
$\bigoplus_{i=1}^n C(\TT^{l_i})$ for some $n\geq 1$ and $l_1,\dots,l_n\in \{0,\dots,k\}$ such that
$\max_{\mu\in IC(\Lambda)} \ell_{\mu}= \max_{i=1}^n l_i$. By the properties in
Section~\ref{sec.1.3}, it follows that
\[
sr(C^*(\Lambda))=sr\Big(\bigoplus_{i=1}^n C(\TT^{l_i})\Big)=\max_{i=1,\dots,n} \lfloor l_i/2\rfloor+1=1.\qedhere
\]
\end{proof}

\begin{remark}
\label{sr1example} It turns out that $C^*$-algebras of finite $k$-graphs with $k>1$ rarely have
stable rank one: the condition $\max_{\mu\in IC(\Lambda)} \ell_{\mu}=1$ is rather strict. As
Proposition~\ref{ell.mu} indicates, if $\Lambda^0$ is finite and $sr(C^*(\Lambda))=1$ (hence
stably finite), then any initial cycle in $\Lambda$ has at most one colour. Using
Lemma~\ref{unit.sum}\eqref{part2.get} and the factorisation property, it follows that any cycle in
$\Lambda$ has at most one colour.

Figure~\ref{pic1.5} illustrates two examples of $2$-graphs $\Lambda$ with $C^*(\Lambda)$ of stable
rank one. The first example, illustrated on the left, has two vertices $v_1,v_2$, a single edge
red (dashed) loop based at $v_1$, and single edge blue  (solid) loop based at $v_2$. The second
example, shown on the right in Figure~\ref{pic1.5}, is different in that it is connected and
contains no loops.
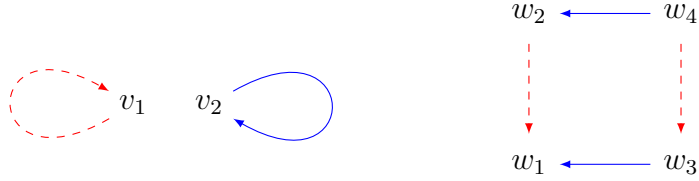
\begin{figure}
\begin{center}
\begin{tikzpicture}

\def \n {1}
\def \radiuss {1.0cm}
\def \radiusm {1.15cm}
\def \radiusb {1.15cm}
  \node[circle, inner sep=0pt] (v1) at (-1,0) {$\ $};
  \node[circle, inner sep=0pt] (v2) at (1, 0) {$\ $};
  \node[inner sep=2.8pt, circle] (27) at (-2.5,0) {$v_1$};
   \node[inner sep=2.8pt, circle] (28) at (-1.5,0) {$v_2$};	
   \path[->,every loop/.style={looseness=14}, >=latex] (27)
			 edge  [in=150,out=210,loop, red, dashed, >=latex] ();
				
   \path[->,every loop/.style={looseness=14}, >=latex] (28)
			 edge  [in=150+180,out=210+180,loop, blue, >=latex] ();
\end{tikzpicture}\ \ \ \ \ \ \ \ \
\begin{tikzpicture}
  \node[inner sep=2.8pt, circle] (27) at (0,0) {$w_1$};
   \node[inner sep=2.8pt, circle] (28) at (0,2) {$w_2$};	
     \node[inner sep=2.8pt, circle] (29) at (2,0) {$w_3$};
   \node[inner sep=2.8pt, circle] (30) at (2,2) {$w_4$};	
      \path[->, red, dashed, >=latex] (28) edge [below] node {} (27);
            \path[->, red, dashed, >=latex] (30) edge [below] node {} (29);

                  \path[->, blue, >=latex] (29) edge [below] node {} (27);
            \path[->, blue, >=latex] (30) edge [below] node {} (28);

\end{tikzpicture}
\caption{Two examples of $2$-graphs with a $C^*$-algebra of stable rank one.} \label{pic1.5}
\end{center}
\end{figure}
\end{remark}

Following \cite{MR2920846}, for $n\in \NN^k$ there is a {shift map} $\sigma^n: \{x \in W_\Lambda :
n\leq d(x) \} \to W_\Lambda$ such that $d(\sigma^n(x))=d(x)-n$ and $\sigma^n(x)(p,q)=x(n+p,n+q)$
for $0\leq p \leq q \leq d(x)-n$ where we use the convention $\infty-a=\infty$ for $a \in \NN$.
For $x \in W_\Lambda$ and $n\leq d(x)$, we then have $x(0,n)\sigma^n(x) = x$. We now show an easy
way to compute $\ell_\mu$, using only the degree of $\mu$.
\begin{prop}
\label{ell.mu} Let $\Lambda$ be a finite, locally convex $k$-graph such that $\Lambda$ has no
cycle with an entrance. Then for each $\mu\in \IC(\Lambda)$,
$$\ell_\mu = |\{i\leq k: d(\mu)_i> 0\}||.$$
\end{prop}

\begin{proof}
Let $I=\{i\leq k: d(\mu)_i> 0\}$. We must show that $\ell_\mu = | I |$. If $I=\emptyset$ then
$\ell_\mu = 0= | I |$, so assume that $I$ is nonempty. By~\eqref{groupofmu},
 $$G_\mu=\{m-n: n,m\leq d(\mu^{\infty}), \mu^{\infty}(m)=\mu^{\infty}(n)\}.$$
Since each $n,m\leq d(\mu^{\infty})$ satisfy $n,m\in \espan_{\NN}\{e_i :  i\in I\}$, the rank of
$G_\mu$ is at most $| I |$. Consequently, it suffices to show $G_\mu$ contains a subgroup of rank
$| I |$.

Let $v\mydef \mu^{\infty}(0)$. We claim that for each colour $i\in I$, there exists a positive
integer $m_i$ such that $\mu^{\infty}(0,m_ie_i)=v$. Indeed, since $\Lambda^0$ is finite there
exists $m<n$ such that $\mu^\infty(m e_i)=\mu^\infty(n e_i)$. Now using that $\mu^\infty\in
\Lambda^{\leq \infty}$ (see Lemma~\ref{lem.init2}) and that for every vertex $w$ on $\mu$ there is
a unique path in $w\Lambda^{\leq \infty}$ (see Section~\ref{sec.ex}) we get that $\sigma^{m
e_i}(\mu^\infty)=\sigma^{n e_i}(\mu^\infty)$. Now for $N\mydef md(\mu)$ it follows that
$\sigma^{N}(\mu^\infty)=\mu^\infty$. Since
$$\mu^\infty=\sigma^{N}(\mu^\infty)=\sigma^{N-me_i+me_i}(\mu^\infty)=\sigma^{N-me_i+ne_i}(\mu^\infty)=\sigma^{(n-m)e_i}(\mu^\infty),$$
we get $(\mu^\infty)((n-m)e_i)=v$. Hence $\mu^\infty$ contains a cycle of degree $(n-m)e_i$ based
at $v$. In particular we can use $m_i\mydef n-m$.

By the preceding paragraph $\{m_ie_i : i\in I\}\subseteq G_\mu$ is a $\mathbb{Z}$-linearly
independent set generating a rank-$|I|$ subgroup of $G_\mu$. So the rank of $G_\mu$ is $| I |$.
\end{proof}

\begin{cor}
\label{sr1.cor} Let $\Lambda$ be a finite, locally convex $k$-graph. Then $sr(C^*(\Lambda))=1$ if
and only if $\Lambda$ has no cycle with an entrance and no initial cycle with more than one
colour.
\end{cor}

\begin{proof}
Combine Proposition~\ref{stably.finite}, Theorem~\ref{sr1}, and Proposition~\ref{ell.mu}.
\end{proof}

\begin{remark}
A \emph{graph trace} on a locally convex row-finite $k$-graph $\Lambda$ is a function
$g\colon\Lambda^0\to \RR^+$ satisfying the {graph trace property}, $g(v)=\sum_{\lambda \in
v\Lambda^{\leq n}} g(s(\lambda))$ for all $v\in \Lambda^0$ and $n\in \NN^k$. It is \emph{faithful}
if it is nonzero on every vertex in $\Lambda$ (\cite{MR2434188, MR1962131}).

It can be shown that Corollary~\ref{sr1.cor} remains valid if we replace ``has no cycle with an
entrance'' by ``admits a faithful graph trace''. Indeed, the $C^*$-algebra of a row-finite and
cofinal $k$-graph $\Lambda$ with no sources is stably finite if and only if $\Lambda$ admits a
faithful graph trace  \cite[Theorem 1.1]{MR3507995}, and for $\Lambda^0$ finite this remains true
without ``cofinal'' and with ``locally convex'' instead of  ``no sources" (by virtue of
Theorem~\ref{direct.sum.thm} and \cite[Lemma~7.1]{MR3311883}).
\end{remark}

\section{Stable rank in the simple and cofinal case.}
\label{four} In this section we focus on stable rank of $k$-graph $C^*$-algebras for which the
$k$-graph is cofinal, corresponding to boxes 1~and~3 in Figure~\ref{fbuha}. Since simple $k$-graph
$C^*$-algebras constitute a sub-case of this situation (as illustrated below), we consider those
first.

Let $\Lambda$ be a row-finite, locally convex $k$-graph. Following \cite{MR2270926}, $\Lambda$ is
\emph{cofinal} if for all pairs $v,w\in \Lambda^0$ there exists $n\in \NN^k$ such that
$s(w\Lambda^{\leq n})\subseteq s(v\Lambda)$.  Following \cite{MR2534246}, $\Lambda$ has
\emph{local periodicity $m,n$ at $v$} if for every $x\in v\Lambda^{\leq\infty}$, we have $m - (m
\wedge d(x)) = n - (n \wedge d(x))$ and $\sigma^{m \wedge d(x)}(x) =\sigma^{n \wedge d(x)}(x)$. If
$\Lambda$ fails to have local periodicity $m,n$ at $v$ for all $m\neq n\in\NN^k$ and
$v\in\Lambda^0$, we say that $\Lambda$ has \emph{no local periodicity}. By
\cite[Theorem~3.4]{MR2534246},
$$\Lambda \text{ is cofinal and has no local periodicity if and only if } C^*(\Lambda) \text{ is simple.}$$

The stable rank of $1$-graph $C^*$-algebras is well understood (see \cite[Theorem~3.4]{MR2001940},
\cite[Theorem~3.3]{MR1868695} and \cite[Theorem~3.1]{MR2059803}), but the following is new for
$k>1$. Recall that a \emph{cycle} is a path $\lambda \in \Lambda \setminus \Lambda^0$ such that
$r(\lambda) = s(\lambda)$.

\begin{prop}
\label{s.r.simple} Let $\Lambda$ be a finite, locally convex $k$-graph. Suppose that $\Lambda$ is
cofinal and has no local periodicity (i.e., $C^*(\Lambda)$ is simple). Then
\[
    sr(C^*(\Lambda))=\begin{dcases}
        1 & \text{if $\Lambda$ contains no cycles}  \\
        \infty & \text{otherwise.} \\
    \end{dcases}
\]
\end{prop}

\begin{proof}
If $\Lambda$ contains no cycles then \cite[Corollary~5.7]{MR2920846} gives $C^*(\Lambda)\cong
M_{\Lambda v}(\CC)$ for some vertex $v \in \Lambda^0$. Using \cite[Proposition~1.7 and
Theorem~3.6]{MR693043} we obtain that $sr(C^*(\Lambda))=1$.

If $\Lambda$ contains a cycle, then another application of \cite[Corollary~5.7]{MR2920846} (see
also \cite[Remark~5.8]{MR4163862}) gives that $C^*(\Lambda)$ is purely infinite. Since
$C^*(\Lambda)$ is unital, simple and purely infinite it contains two isometries with orthogonal
ranges, so \cite[Proposition~6.5]{MR693043} gives $sr(C^*(\Lambda))=\infty$.
\end{proof}

In conclusion, the stable rank of a unital simple $k$-graph $C^*$-algebra is completely determined
by presence or absence of a cycle in the $k$-graph.

\subsection{The cofinal case}
We now consider the cofinal case. We start by recalling a result of Jeong, Park and Shin about
directed graphs (or $1$-graphs). We refer to \cite{MR1868695}  for the terminology involved.
\begin{prop}[{\cite[Proposition~3.7]{MR1868695}}]
\label{loc.fin.cofinal} Let $E$ be a locally finite directed graph. If $E$ is cofinal then either
$sr(C^*(E))= 1$ or $C^*(E)$ is purely infinite simple.
\end{prop}

\begin{remark}
\label{counter.example} We illustrate why for $k$-graphs we can not hope for a result similar to
Proposition~\ref{loc.fin.cofinal}.
\begin{figure}
\begin{center}
\begin{tikzpicture}
  \node[circle, inner sep=0pt] (v1) at (-2.5,0) {$e $};
  \node[circle, inner sep=0pt] (v2) at (0.5, 0) {$a$};
    \node[circle, inner sep=0pt] (v2) at (-0.5,-1) {$b$};
   \node[inner sep=2.8pt, circle] (28) at (-1,0) {$v$};	
   \path[->,every loop/.style={looseness=14}, >=latex] (28)
			 edge  [in=150,out=210,loop, red, dashed, >=latex] ();
				
   \path[->,every loop/.style={looseness=14}, >=latex] (28)
			 edge  [in=150+90,out=210+90,loop, blue, >=latex] ();
			
   \path[->,every loop/.style={looseness=14}, >=latex] (28)
			 edge  [in=150+180,out=210+180,loop, blue, >=latex] ();			
			 		
\end{tikzpicture}\ \ \ \ \ \ \ \ \
\caption{Example of $2$-graph with a $C^*$-algebra of stable rank infinity.}\label{non-cof}
\end{center}
\end{figure}
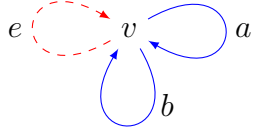
Consider the $2$-graph $\Lambda$ in Figure~\ref{non-cof} with two blue edges $a,b\in
\Lambda^{e_1}$ and one red edge $e\in \Lambda^{e_2}$ and the factorisation property $ae=ea$,
$be=eb$. Since $\Lambda$ has only one vertex, it is automatically cofinal. However, $C^*(\Lambda)$
neither has stable rank one nor is purely infinite simple as the following discussion shows:

The $C^*$-algebra $C^*(\Lambda)$ fails to have stable rank one because it is not stably finite
(containing a cycle with an entrance). It is not simple, so in particular not purely infinite
simple because \begin{equation} \label{loc.new} \textrm{for every }x \in  v\Lambda^{\leq
\infty}\textrm{ we have } \sigma^{e_2}(x)=x,
\end{equation}
so $\Lambda$ has local periodicity $p=e_2, q=0 $ at $v$ and $C^*(\Lambda)$ is non-simple.

Because of our particular choice of factorisation rules $ae=ea$, $be=eb$ Lemma~$\ref{examples12}$
implies that $C^*(\Lambda)\cong \mathcal{O}_2\otimes C(\TT)$. If we instead used the factorisation
$ae=eb$, $be=ea$, then Lemma~$\ref{examples12}$ would not apply, but we would still have
$\sigma^{2e_2}(x)=x$ for each $x\in v\Lambda^{\leq \infty}$ making $C^*(\Lambda)$ non-simple.
\end{remark}

Remark~\ref{counter.example} notwithstanding, we are able to provide a characterisation of stable
rank in the cofinal case. Given a $C^*$-algebra $A$, we write $a\oplus b$ for the diagonal matrix
$\textrm{diag}(a,b)$ in $M_2(A)$ and write $\sim$ for the von Neumann equivalence relation between
elements in matrix algebras over $A$. A unital $C^*$-algebra $A$ is \emph{properly infinite} if
$1\oplus 1\oplus r \sim 1$ for some projection $r$ in some matrix algebra over $A$ (for more
details see \cite{MR1878881}).

\begin{thm}
\label{thm.cofinal} Let $\Lambda$ be a cofinal, finite, locally convex $k$-graph. Suppose that
$\Lambda$ contains a cycle with an entrance. Then $C^*(\Lambda)$ is properly infinite and has
stable rank $\infty$.
\end{thm}

\begin{proof}
Let $\mu$ be a cycle with an entrance $\tau$, that is
\[
\tau\in r(\mu)\Lambda, \qquad d(\tau)\leq d(\mu^\infty), \qquad\text{and}\qquad \tau\neq \mu^\infty(0,d(\tau)).
\]
Fix $n\geq 1$ such that $m:=d(\mu)n\geq d(\tau)$. Since $\tau\neq
\mu^\infty(0,d(\tau))=\mu^n(0,d(\tau))$, there exists $\beta\in s(\tau)\Lambda$ such that $\mu^n$
and $\tau\beta$ are distinct elements of $r(\mu)\Lambda^{\leq m}$. Write $r(\mu)\Lambda^{{\leq
m}}=\{\nu_1, \dots, \nu_N\}$ with $\nu_1=\mu^n$ and $\nu_2=\tau\beta$. For each $i=1,\dots,N$ set
$v_i=s(\nu_i)$ and let $x=(s_{\nu_1}, \dots, s_{\nu_N})$.  Then $xx^*= \sum_{\lambda \in
{v_1}\Lambda^{{\leq m}}} s_\lambda s_\lambda^*=s_{v_1}$. Moreover for $i\neq j$, $
s_{\nu_i}^*s_{\nu_j}=0$, so
$$s_{v_1}=xx^*\sim x^*x=\textrm{diag}(s_{\nu_1}^*s_{\nu_1}, \dots, s_{\nu_N}^*s_{\nu_N})=s_{v_1}\oplus\cdots\oplus s_{v_N}.$$

We claim that for any pair of vertices $u,v \in \Lambda^0$ there exist a constant $M_{u,v}$ and a
projection $p_{u,v}$ in some matrix algebra over $C^*(\Lambda)$ such that
\begin{equation}
\label{v.sim.u}
\Big(\bigoplus_{l=1}^{M_{u,v}} s_{u}\Big) \sim s_{v} \oplus p_{u,v}.
\end{equation}

To see this, fix $u,v\in \Lambda^0$. Since $\Lambda$ is cofinal there exists $n\in \NN^k$ such
that $s(v\Lambda^{\leq n})\subseteq s(u\Lambda)$.  Writing $v\Lambda^{\leq n}=\{\mu_1, \dots,
\mu_{M_{u,v}}\}$ and $u_i=s(\mu_i)$, we have $s_{v}\sim s_{u_1}\oplus\cdots\oplus
s_{u_{{M_{u,v}}}}$. Since $s(v\Lambda^{\leq n})=\{u_i: i\leq {M_{u,v}}\}\subseteq s(u\Lambda)$,
for each $i\leq M_{u,v}$ there exists $\lambda_i\in u\Lambda$ such that $s(\lambda_i)=u_i$. Let
$m_i=d(\lambda_i)$ for each $i$. Then for each $i=1,\dots,{M_{u,v}}$,
$$s_u=\sum_{\lambda\in u\Lambda^{\leq m_i}}s_\lambda s_\lambda^*\sim s_{u_i}\oplus p_i$$
for some projection $p_i$ in a matrix algebra over $C^*(\Lambda)$. With $p_{u,v}=\bigoplus_i p_i$
we obtain
\[
\Big(\bigoplus_{l=1}^{M_{u,v}} s_{u}\Big) \sim \Big(\bigoplus_{i=1}^{M_{u,v}} s_{u_i}\Big)\oplus p_{u,v}\sim s_{v} \oplus p_{u,v},
\]
which establishes the claim.

Applying \eqref{v.sim.u} to $u=v_2$ and $v=v_1$ we get
$$\Big(\bigoplus_{l=1}^{M_{v_2,v_1}} s_{v_2}\Big)\sim s_{v_1} \oplus p_{v_2,v_1}.$$
Recall that $s_{v_1}\sim s_{v_1}\oplus s_{v_2}\oplus \big(\bigoplus_{i=3}^N s_{v_i}\big)$. Let
$q\mydef  p_{v_2,v_1}\oplus \big(\bigoplus_{l=1}^{M_{v_2,v_1}}\bigoplus_{i=3}^N s_{v_i}\big)$,
meaning that if $N=2$ then $q=  p_{v_2,v_1} \oplus 0$. Then
\begin{align}
\label{v1.prop.inf}
s_{v_1}\sim s_{v_1}\oplus \Big(\bigoplus_{l=1}^{M_{v_2,v_1}} s_{v_2}\Big)\oplus \Big(\bigoplus_{l=1}^{M_{v_2,v_1}} \bigoplus_{i=3}^N s_{v_i}\Big)\sim s_{v_1}\oplus s_{v_1}\oplus q.
\end{align}
Applying \eqref{v.sim.u} to $u=v_1$ and to each $v\in \Lambda^0\setminus \{v_1\}$ at the second
equality, and putting $L := 2 + \sum_{v \in \Lambda^0\setminus\{v_1\}} M_{v_1,v}$, we calculate:
\begin{align*}
1\oplus 1 \oplus \Big(\bigoplus_{v\in \Lambda^0\setminus \{v_1\}} p_{v_1,v} \Big)
    &\sim 1\oplus s_{v_1} \oplus \Big(\bigoplus_{v\in \Lambda^0\setminus \{v_1\}}(s_{v}\oplus  p_{v_1,v} )\Big) \\
    &\sim 1\oplus s_{v_1} \oplus \Big(\bigoplus_{v\in \Lambda^0\setminus \{v_1\}}\Big(\bigoplus_{i=1}^{M_{v_1,v}} s_{v_1}\Big)\Big)\\
    &\sim \Big(\bigoplus_{v\in \Lambda^0\setminus \{v_1\}}s_v\Big)\oplus s_{v_1} \oplus s_{v_1} \oplus \Big(\bigoplus_{v\in \Lambda^0\setminus \{v_1\}}\Big(\bigoplus_{i=1}^{M_{v_1,v}} s_{v_1}\Big)\Big)\\
    &\sim \Big(\bigoplus_{v\in \Lambda^0\setminus \{v_1\}}s_v\Big)\oplus \Big(\bigoplus_{j=1}^L s_{v_1}\Big).
\end{align*}
Using \eqref{v1.prop.inf} we have $s_{v_1}\sim s_{v_1}\oplus \big(\bigoplus_{j=1}^{L-1}
s_{v_1}\big) \oplus \big(\bigoplus_{j=1}^{L-1} q\big)$, so $r= \big(\bigoplus_{v\in
\Lambda^0\setminus \{v_1\}} p_{v_1,v} \big) \oplus \big(\bigoplus_{j=1}^{L-1} q\big)$ satisfies
$$1\oplus 1\oplus r \sim 1.$$
Hence $1$ is properly infinite. Now \cite[Proposition~6.5]{MR693043} gives
$sr(C^*(\Lambda))=\infty$.\end{proof}

With Proposition~\ref{loc.fin.cofinal} in mind, the following is a dichotomy for the
$C^*$-algebras associated to cofinal finite $k$-graphs.

\begin{cor}
\label{cor.4.5} Let $\Lambda$ be a cofinal, finite, locally convex $k$-graph. Then either
$C^*(\Lambda)$ is stably finite and $sr(C^*(\Lambda))$ is given by Corollary~\ref{stab.fin}, or
$C^*(\Lambda)$ is properly infinite and $sr(C^*(\Lambda))=\infty$.
\end{cor}

\begin{proof}
If $C^*(\Lambda)$ is not stably finite then $\Lambda$ contains a cycle with an entrance by
Proposition~\ref{stably.finite}. Hence $C^*(\Lambda)$ is properly infinite and
$sr(C^*(\Lambda))=\infty$ by Theorem~\ref{thm.cofinal}. Conversely, if  $C^*(\Lambda)$ is properly
infinite, then it is also infinite, so $C^*(\Lambda)$ is not finite and hence not stably finite.
If $C^*(\Lambda)$ is stably finite then Corollary~\ref{stab.fin} applies.
\end{proof}

\begin{remark}
Theorem~\ref{thm.cofinal} includes cases not covered by any of our preceding results. Consider for
example the $2$-graph $\Lambda$ in Figure~\ref{non-cof}. By Theorem~\ref{thm.cofinal} the
associated $C^*$-algebra has stable rank infinity.
\end{remark}

\begin{example}
By Corollary~\ref{cor.4.5}, we can compute the stable rank of $k$-graph $C^*$-algebras in boxes
1~to~3 in Figure~\ref{fbuha}. It therefore makes sense to consider the range of stable rank
achieve by these $C^*$-algebra. In box~3 stable rank infinity can be obtained as in
Remark~\ref{counter.example}. For finite stable rank, Table~\ref{table.2} lists a few $4n$-graphs
$\Lambda$ together with their associated $C^*$-algebra and its stable rank (we use a multiple of 4
because it makes the formulas for the stable rank simpler).
\begin{table}
\caption{A few examples of $4n$-graphs.}\label{table.2}
\setlength{\tabcolsep}{5mm} 
\def\arraystretch{1.25} 
\centering
  \begin{tabular}{|c|c|c|c|}
      \hline

      $4n$-graph $\Lambda$  & $|\Lambda^{\leq N}|$  &   $C^*(\Lambda)$    &  $sr( C^*(\Lambda))$      \\ \hline

\begin{minipage}{0.13\textwidth}\begin{tikzpicture}
\def \margin {6}
  \node[circle, inner sep=0pt] (v0) at (-1.8,.8) {$4n$};
  \node[inner sep=2.8pt, circle] (27) at (-2.5,0) {$v_1$};	
   \path[->,every loop/.style={looseness=12}, >=latex] (27)
			 edge  [in=150-90,out=210-90,loop, black, >=latex] ();
\end{tikzpicture}\end{minipage}   &1 &    $C(\TT^{4n})$  & $2n+1$    \\ \hline

\begin{minipage}{0.25\textwidth}\begin{tikzpicture}
\def \margin {6}
  \node[circle, inner sep=0pt] (v0) at (-1.8,.8) {$4n$};
  \node[inner sep=2.8pt, circle] (27) at (-2.5,0) {$v_1$};	
    \node[inner sep=2.8pt, circle] (28) at (-5,0) {$w_1$};	
   \path[->,every loop/.style={looseness=12}, >=latex] (27)
			 edge  [in=150-90,out=210-90,loop, black, >=latex] ();
   \path[->] (27) edge [below] node  {$4n$} (28);
\end{tikzpicture}\end{minipage}  &2  &     $M_2(C(\TT^{4n}))$   & $n+1$    \\ \hline
\begin{minipage}{0.25\textwidth}\begin{tikzpicture}
\def \margin {6}
  \node[circle, inner sep=0pt] (v0) at (-1.8,.8) {$4n$};
  \node[inner sep=2.8pt, circle] (27) at (-2.5,0) {$v_1$};	
    \node[inner sep=2.8pt, circle] (28) at (-5,0) {$w_1$};	
       \node[inner sep=2.8pt, circle] (29) at (-5,.9) {$w_2$};	
   \path[->,every loop/.style={looseness=12}, >=latex] (27)
			 edge  [in=150-90,out=210-90,loop, black, >=latex] ();
   \path[->] (27) edge [below] node  {$4n$} (28);
      \path[->] (27) edge [above] node  {$4n$} (29);
\end{tikzpicture} \end{minipage}  &3  &     $M_3(C(\TT^{4n}))$   &  $ \displaystyle\left\lceil \frac{2n}{3}\right\rceil+1$        \\ \hline

\begin{minipage}{0.25\textwidth}\begin{tikzpicture}
\def \margin {6}
  \node[circle, inner sep=0pt] (v0) at (-1.8,.8) {$4n$};
  \node[inner sep=2.8pt, circle] (27) at (-2.5,0) {$v_1$};	
    \node[inner sep=2.8pt, circle] (28) at (-5,0) {$w_1$};	
           \node[inner sep=2.8pt, circle] (29b) at (-5,.6) {$\vdots$};
                    \node[inner sep=2.8pt, circle] (29b) at (-3.8,.35) {$\vdots$};
       \node[inner sep=2.8pt, circle] (29) at (-5,1) {$w_m$};	
   \path[->,every loop/.style={looseness=12}, >=latex] (27)
			 edge  [in=150-90,out=210-90,loop, black, >=latex] ();
   \path[->] (27) edge [below] node  {$4n$} (28);
      \path[->] (27) edge [above] node  {$4n$} (29);
\end{tikzpicture} \end{minipage}   &$m+1$ &     $M_{m+1}(C(\TT^{4n}))$  & $ \displaystyle\left\lceil \frac{2n}{m+1}\right\rceil+1$   \\ \hline

\begin{minipage}{0.25\textwidth}\begin{tikzpicture}
\def \margin {6}
  \node[circle, inner sep=0pt] (v0) at (-1.8,.8) {$4n$};
  \node[inner sep=2.8pt, circle] (27) at (-2.5,0) {$v_1$};	
    \node[inner sep=2.8pt, circle] (28) at (-5,0) {$v_3$};	
        \node[inner sep=2.8pt, circle] (28b) at (-3.75,0) {$v_2$};
   \path[->,every loop/.style={looseness=12}, >=latex] (27)
			 edge  [in=150-90,out=210-90,loop, black, >=latex] ();
   \path[->] (27) edge [below] node  {$4n$} (28b);
      \path[->] (28b) edge [below] node  {$4n$} (28);
\end{tikzpicture}  \end{minipage}  &  ${4n\choose 2}$ &     $M_{{4n\choose 2}}(C(\TT^{4n}))$   & $\displaystyle \left\lceil \frac{2n}{{4n\choose 2}}\right\rceil+1$     \\ \hline

\begin{minipage}{0.30\textwidth}\begin{tikzpicture}
\def \margin {6}
  \node[circle, inner sep=0pt] (v0) at (-1.8,.8) {$4n$};
\node[circle, inner sep=0pt] (v0) at (-4.3,.8) {$2n$};
  \node[inner sep=2.8pt, circle] (27) at (-2.5,0) {$v_1$};	
    \node[inner sep=2.8pt, circle] (28) at (-5,0) {$v_2$};	
   \path[->,every loop/.style={looseness=12}, >=latex] (27)
			 edge  [in=150-90,out=210-90,loop, black, >=latex] ();
   \path[->,every loop/.style={looseness=12}, >=latex] (28)
			 edge  [in=150-90,out=210-90,loop, black, >=latex] ();
      \path[->] (27) edge [below] node  {$2n$} (28);
\end{tikzpicture}\end{minipage}   &2 &   $M_{2}(C(\TT^{4n}))$,  & $n+1$    \\ \hline
  \end{tabular}
\end{table}

Except for the last $4n$-graph, each black edge represents exactly $4n$ edges of different
colours, one of each colour; the last $4n$-graph has $2n$ loops at $v_2$, one each of the first
$2n$ colours and $2n$ edges from $v_1$ to $v_2$, one each of the remaining $2n$ colours. Each
example admits a unique factorisation rule, so each illustration in Table~\ref{table.2} represents
a unique $4n$-graph.
\end{example}

\section{Stable rank in the non-stably finite, non-cofinal case}
\label{five} So far we have looked at the stably finite case (including stable rank one) and the
cofinal case (including the simple case). Here we study the remaining case corresponding to box 4
in Figure~\ref{fbuha}.

We start by revisiting the cofinality condition for row-finite locally convex $k$-graphs.
Following \cite{MR2534246}, a subset $H\subseteq \Lambda^0$ is \emph{hereditary} if
$s(H\Lambda)\subseteq H$. We say $H$ is \emph{saturated} if for all $v\in \Lambda^0$,
\[
 \{s(\lambda):\lambda\in v\Lambda^{\leq e_i}\}\subseteq H
    \text{ for some } i\in\{1,\dots,k\} \Longrightarrow v\in H.
\]
or equivalently, if $v\not\in H$ implies that for each $n\in \NN^k$, $s(v\Lambda^{\leq
n})\not\subseteq H$ (see Lemma~\ref{saturated}). The relevant characterisation of cofinal is
included in Lemma~\ref{cofinal.result} below with a short proof based on \cite{MR2670219} and
\cite{MR2270926}. Since this paper focuses on unital $k$-graph $C^*$-algebras,  it is worth
pointing out that Lemmas \ref{saturated}~and~\ref{cofinal.result} do not assume that
$|\Lambda^0|<\infty$.

\begin{lemma}
\label{saturated} Let $\Lambda$ be a row-finite locally convex $k$-graph. Then $H\subseteq
\Lambda^0$ is saturated if and only if for all $v\in \Lambda^0$, $v\not\in H$ implies that for
each $n\in \NN^k$, $s(v\Lambda^{\leq n})\not\subseteq H$.
\end{lemma}
\begin{proof}
Fix $v\in \Lambda^0$. Suppose $v\not\in H$. Since $\Lambda$ is saturated, for all $i\leq k$,
$\{s(\lambda):\lambda\in v\Lambda^{\leq e_i}\}\not\subseteq H$. Clearly $s(v\Lambda^{\leq
m})\not\subseteq H$ for $m=0$. Fix any $m\in \NN^k\setminus\{0\}$. Set
$(n^{(0)},v^{(0)},\lambda^{(0)})=(m,v,v)$. Choose $i$ such that $n^{(0)}_i\neq 0$. Since
$v^{(0)}\not\in H$, there exists $\mu^{(1)}\in v^{(0)}\Lambda^{\leq e_i}\setminus \Lambda H$. Set
 $(n^{(1)},v^{(1)},\lambda^{(1)})=(n^{(0)}-e_i,s(\mu^{(1)}),\lambda^{(0)}\mu^{(1)})$.
Choose $i$ such that $n^{(1)}_i\neq 0$. Since $v^{(1)}\not\in H$, there exists $\mu^{(2)}\in
v^{(1)}\Lambda^{\leq e_i}\setminus \Lambda H$. Set
$(n^{(2)},v^{(2)},\lambda^{(2)})=(n^{(1)}-e_i,s(\mu^{(2)}),\lambda^{(1)}\mu^{(2)})$.

For each step, $|n^{(i)}|=|m|-i$, so $l=|m|$ satisfies $n^{(l)}=0$. Notice that $\lambda^{(0)}\in
v\Lambda^{\leq (m-n^{(0)})}, \lambda^{(1)}\in v\Lambda^{\leq (m-n^{(1)})}, \dots, \lambda^{(l)}\in
v\Lambda^{\leq (m-n^{(l)})}$. Hence $\lambda^{(l)}\in v\Lambda^{\leq m}$ and
$s(\lambda^{(l)})\not\in H$ so $s(v\Lambda^{\leq m})\not\subseteq H$.
\end{proof}

\begin{lemma}[{\cite{MR2670219, MR2270926}}]
\label{cofinal.result} Let $\Lambda$ be a row-finite locally convex $k$-graph. Then the following
are equivalent:
\begin{enumerate}
\item \label{cif.1.1}$\Lambda$ is cofinal;
\item \label{cif.2.1} for all $v\in \Lambda^0$, and $(\lambda_i)$ with $\lambda_i\in
    \Lambda^{\leq (1,\dots, 1)}$, and $s(\lambda_i)=r(\lambda_{i+1})$ there exist $i\in \NN$ and
    $n\leq d(\lambda_i)$ such that $v\Lambda \lambda_i(n) \neq \emptyset$; and
\item \label{cif.3} $\Lambda^0$ contains no nontrivial hereditary saturated subsets.
\end{enumerate}
\end{lemma}
\begin{proof}
Firstly we show that \eqref{cif.1.1}$\Rightarrow$\eqref{cif.3}. Suppose \eqref{cif.1.1} and
suppose that $H\subseteq \Lambda^0$ is a nonempty hereditary, saturated set. We show that
$H=\Lambda^0$. Fix $v\in \Lambda^0$. Since $H$ is nonempty, there exists $w\in H$. By
\eqref{cif.1.1} there exists $n\in \NN^k$ such that $s(v\Lambda^{\leq n})\subseteq s(w\Lambda)$.
Since $H$ is hereditary, $s(v\Lambda^{\leq n})\subseteq s(H\Lambda) \subseteq H$. Hence
Lemma~\ref{saturated} gives $v\in H$.

Now we show that \eqref{cif.3}$\Rightarrow$\eqref{cif.2.1}. Suppose that \eqref{cif.2.1} fails,
that is, there exist $v\in \Lambda^0$, and a sequence $(\lambda_i)$ with $\lambda_i\in
\Lambda^{\leq (1,\dots, 1)}, s(\lambda_i)=r(\lambda_{i+1})$ for all $i$ such that for all $i\in
\NN$ and all $n\leq d(\lambda_i)$ we have $v\Lambda \lambda_i(n) = \emptyset$. Let
\[
    H=\{w\in \Lambda^0: w\Lambda \lambda_i(n) = \emptyset \text{ for all } i\in \NN \text{ and }n\leq d(\lambda_i)\}.
\]
Then $H$ is nontrivial as $v\in H$ and hereditary because if $u\Lambda w\neq \emptyset$ then
$s(w\Lambda)\subseteq s(u\Lambda)$. To show that $H$ is saturated take $u \in \Lambda^0$ and
$j\leq k$ such that $s(u{\Lambda}^{\le e_j})\subseteq H$. We must show that $u\in H$. Assume
otherwise for contradiction. We have $u\not\in s(u{\Lambda}^{\le e_j})$ because otherwise $u=s(u)$
belongs to $H$, so $u{\Lambda}^{e_j}\neq \emptyset$. Since $u\not\in H$, there exists $\lambda\in
u\Lambda$ such that $s(\lambda)=\lambda_i(n)$ for some $i,n$. We claim that $d(\lambda)_j=0$.
Indeed, if not, then $\lambda=\mu\mu'$ for some $\mu\in u{\Lambda}^{e_j}$. We then have $s(\mu)\in
s(u{\Lambda}^{e_j})\subseteq H$, so $s(\mu)\Lambda \lambda_i(n) = \emptyset$ contradicting
$\mu'\in s(\mu)\Lambda \lambda_i(n)$. Since $\Lambda$ is locally convex and $d(\lambda)_j=0$ and
since $u\Lambda^{e_j}\neq \emptyset$, we have
$\lambda_i(n)\Lambda^{e_j}=s(\lambda)\Lambda^{e_j}\neq \emptyset$. Let $\beta=\lambda_i(n,
d(\lambda_i))$. Since $\Lambda$ is locally convex, either $d(\beta)_j\neq 0$ or
$s(\lambda_i)\Lambda^{e_j}\neq \emptyset$. Since $\lambda_{i+1}\in\Lambda^{\leq (1,\dots,1)}$ it
follows that $d(\beta\lambda_{i+1})\geq e_j$. Now $\lambda':=\lambda\beta\lambda_{i+1} \in
u\Lambda$ satisfies $s(\lambda')=\lambda_{i'}({n'})$ for some ${i'},{n'}$. But then, just as we got
$d(\lambda)_j=0$, we deduce $d(\lambda')_j=0$, a contradiction. So $H$ is saturated, so \eqref{cif.3}
does not hold.

Finally we prove \eqref{cif.2.1}$\Rightarrow$\eqref{cif.1.1}. Given \eqref{cif.2.1}, we suppose
that \eqref{cif.1.1} fails, and we derive a contradiction. Since \eqref{cif.1.1} fails, there
exist $v,w\in \Lambda^0$ such that for all $n\in \NN^k$, we have $s(w\Lambda^{\leq
n})\not\subseteq s(v\Lambda)$. Set
$$K=\{u\in \Lambda^0: s(u\Lambda^{\leq n})\not\subseteq s(v\Lambda) \text{ for all } n\in \NN^k\}.$$
Fix $u\in K$ and $j\leq k$. We claim that there exists $\mu\in u\Lambda^{\leq e_j}$ such that
$s(\mu)\in K$. Indeed if $s(u\Lambda^{\leq e_j})\subseteq \Lambda^0\setminus K$, then for each
$\mu\in u\Lambda^{\leq e_j}$ there exists $n_\mu\in \NN^k$ such that
$s(\mu\Lambda^{n_\mu})\subseteq s(v\Lambda)$. Since $s(v\Lambda)$ is hereditary, it follows that
$n=\bigvee_{\mu\in u\Lambda^{\leq e_j}} n_\mu$ satisfies
$s(u\Lambda^{\leq{n+e_j}})=\bigcup_{\mu\in u\Lambda^{\leq e_j}}s(\mu \Lambda^{\leq{n}}) \subseteq
s(v\Lambda)$, contradicting $u\in K$.

Since $w\in K$ we can construct a sequence $(\lambda_i)$ such that each $\lambda_i\in
\Lambda^{\leq (1,\dots, 1)}$, each $s(\lambda_i)=r(\lambda_{i+1})$, and for each $n\leq
d(\lambda_i)$ we have $\lambda_i(n)\in K$. By \eqref{cif.2.1} there exist $i$ and $n\leq
d(\lambda_i)$ such that $v\Lambda \lambda_i(n) \neq \emptyset$, i.e., such that
$s(\lambda_i(n)\Lambda^{\leq 0})\subseteq s(v\Lambda)$. So $\lambda_i(n)\not\in K$, a
contradiction.
\end{proof}

\begin{remark}
When a $k$-graph $\Lambda$ has only one vertex, it is automatically cofinal, and we deduce that
the stable rank of $C^*(\Lambda)$ is infinite if there exists $j \le k$ such that $|\Lambda^{e_j}|
\ge 2$, and is equal to $\lfloor k/2\rfloor + 1$ if each $|\Lambda^{e_j}| = 1$.
\end{remark}

\begin{remark}
\label{unknown} We now present all the $2$-graphs $\Lambda$ with $|\Lambda^0|=2$ for which we have
been unable to compute the stable rank of the associated $C^*$-algebra $C^*(\Lambda)$ (see
Figure~\ref{special.graphs.2}). In each case the $2$-graph $\Lambda$ fails to be cofinal, because
$\Lambda^0$ contains one nontrivial hereditary saturated subset, denoted $H$.

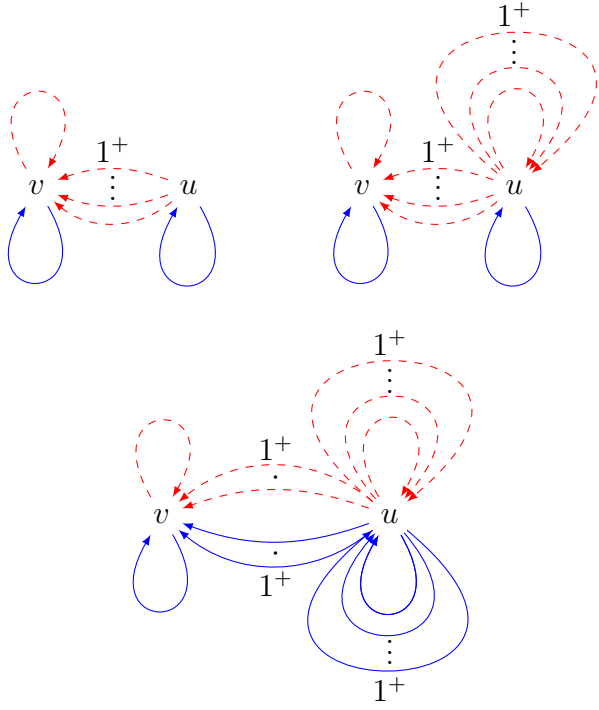
\begin{figure}
\begin{center}
\begin{tikzpicture}

  \node[circle, inner sep=0pt] (v1) at (-2,.1) {$\vdots$};
  \node[circle, inner sep=0pt] (v2) at (-2, .5) {$1^+$};
   \node[inner sep=2.8pt, circle] (28) at (-1,0) {$u$};	
      \node[inner sep=2.8pt, circle] (29) at (-3,0) {$v$};	

       \draw[-latex, red, dashed] (28) edge[out=160,in=20] (29);
    \draw[-latex, red, dashed] (28) edge[out=200,in=340] (29);
        \draw[-latex, red, dashed] (28) edge[out=220,in=320] (29);

   \path[->,every loop/.style={looseness=14}, >=latex] (29)
			 edge  [in=150-90,out=210-90,loop, red, dashed, >=latex] ();
				
   \path[->,every loop/.style={looseness=14}, >=latex] (28)
			 edge  [in=150+90,out=210+90,loop, blue, >=latex] ();
			
   \path[->,every loop/.style={looseness=14}, >=latex] (29)
			 edge  [in=150+90,out=210+90,loop, blue, >=latex] ();			
			 		
\end{tikzpicture}
\begin{tikzpicture}
  \node[circle, inner sep=0pt] (v1) at (-2,.1) {$\vdots$};
  \node[circle, inner sep=0pt] (v2) at (-2, .5) {$1^+$};
    \node[circle, inner sep=0pt] (v2) at (-1,1.9) {$\vdots$};
      \node[circle, inner sep=0pt] (v2) at (-1, 2.3) {$1^+$};
   \node[inner sep=2.8pt, circle] (28) at (-1,0) {$u$};	
      \node[inner sep=2.8pt, circle] (29) at (-3,0) {$v$};	

       \draw[-latex, red, dashed] (28) edge[out=160,in=20] (29);
    \draw[-latex, red, dashed] (28) edge[out=200,in=340] (29);
        \draw[-latex, red, dashed] (28) edge[out=220,in=320] (29);

   \path[->,every loop/.style={looseness=14}, >=latex] (29)
			 edge  [in=150-90,out=210-90,loop, red, dashed, >=latex] ();
			
   \path[->,every loop/.style={looseness=14}, >=latex] (28)
			 edge  [in=150-90,out=210-90,loop, red, dashed, >=latex] ();
			
   \path[->,every loop/.style={looseness=16}, >=latex] (28)
			 edge  [in=150-100,out=210-80,loop, red, dashed, >=latex] ();
			
   \path[->,every loop/.style={looseness=22}, >=latex] (28)
			 edge  [in=150-110,out=210-70,loop, red, dashed, >=latex] ();
				
   \path[->,every loop/.style={looseness=14}, >=latex] (28)
			 edge  [in=150+90,out=210+90,loop, blue, >=latex] ();
			
   \path[->,every loop/.style={looseness=14}, >=latex] (29)
			 edge  [in=150+90,out=210+90,loop, blue, >=latex] ();			
			 		
\end{tikzpicture}
\begin{tikzpicture}
  \node[circle, inner sep=0pt] (v1) at (-1.5,.5) {$\cdot$};
  \node[circle, inner sep=0pt] (v2) at (-1.5, .9) {$1^+$};
    \node[circle, inner sep=0pt] (v1) at (-1.5,-.5) {$\cdot$};
  \node[circle, inner sep=0pt] (v2) at (-1.5,- .9) {$1^+$};
    \node[circle, inner sep=0pt] (v2) at (0,1.9) {$\vdots $};
      \node[circle, inner sep=0pt] (v2) at (0, 2.3) {$1^+$};
         \node[circle, inner sep=0pt] (v2) at (0,-1.7) {$\vdots $};
      \node[circle, inner sep=0pt] (v2) at (0, -2.3) {$1^+$};
   \node[inner sep=2.8pt, circle] (28) at (0,0) {$u$};	
      \node[inner sep=2.8pt, circle] (29) at (-3,0) {$v$};	

             \draw[-latex, red, dashed] (28) edge[out=140,in=40] (29);
       \draw[-latex, red, dashed] (28) edge[out=160,in=20] (29);
    \draw[-latex, blue] (28) edge[out=200,in=340] (29);
        \draw[-latex, blue] (28) edge[out=220,in=320] (29);

   \path[->,every loop/.style={looseness=14}, >=latex] (28)
			 edge  [in=150-90,out=210-90,loop, red, dashed, >=latex] ();
			
   \path[->,every loop/.style={looseness=16}, >=latex] (28)
			 edge  [in=150-100,out=210-80,loop, red, dashed, >=latex] ();
			
   \path[->,every loop/.style={looseness=22}, >=latex] (28)
			 edge  [in=150-110,out=210-70,loop, red, dashed, >=latex] ();

   \path[->,every loop/.style={looseness=14}, >=latex] (28)
			 edge  [in=150+90,out=210-90+180,loop, blue, >=latex] ();
			
   \path[->,every loop/.style={looseness=16}, >=latex] (28)
			 edge  [in=150-100+180,out=210-80+180,loop, blue, >=latex] ();
			
   \path[->,every loop/.style={looseness=22}, >=latex] (28)
			 edge  [in=150-110+180,out=210-70+180,loop, blue, >=latex] ();
				
   \path[->,every loop/.style={looseness=14}, >=latex] (28)
			 edge  [in=150+90,out=210+90,loop, blue, >=latex] ();
			
   \path[->,every loop/.style={looseness=14}, >=latex] (29)
			 edge  [in=150-90,out=210-90,loop, red, dashed, >=latex] ();				
			
   \path[->,every loop/.style={looseness=14}, >=latex] (29)
			 edge  [in=150+90,out=210+90,loop, blue, >=latex] ();			
			 		
\end{tikzpicture}
\caption{Example of $2$-graphs $\Lambda$ with $C^*(\Lambda)$ of stable rank two or thee.}\label{special.graphs.2}
\end{center}
\end{figure}

In Figure~\ref{special.graphs.2}, for each 2-graph $\Lambda$ the $C^*$-algebra $C^*(\Lambda)$ is
non-simple with $H=\{u\}$. In the first case, we have $C^*(H\Lambda) \cong C(\TT)$, which has
stable rank~1, and so $I_H$ has stable rank~1 because stable rank~1 is preserved by stable
isomorphism. In the remaining two cases, if there is one loop of each colour at $u$ then
$C^*(H\Lambda) \cong C(\TT^2)$ has stable rank~2, and otherwise, Theorem~\ref{thm.cofinal} implies
that $C^*(H\Lambda)$ has stable rank~$\infty$; either way, since $I_H \cong C^*(H\Lambda) \otimes
\Kk$\footnote{To see this, let $X$ be the set $\{u\} \cup \{\mu f : f \in v\Lambda^{e_2} u\text{
and } \mu \in \Lambda^{\NN e_2}v\}$. Use the factorisation property and the Cuntz--Krieger
relations to see that $I_H = \overline{\operatorname{span}}\{s_\mu a s^*_\nu : \mu,\nu \in X\text{
and } a \in s_u C^*(\Lambda) s_u\}$. It is routine that for any finite subset $F \subseteq X$, the
set $\{s_\mu s^*_\nu : \mu,\nu \in F\}$ is a system of matrix units. So Lemma~\ref{matrix.system} gives
$\overline{\operatorname{span}}\{s_\mu a s^*_\nu : \mu,\nu \in F\text{ and } a \in s_u
C^*(\Lambda) s_u\} \cong s_u C^*(\Lambda) s_u \otimes M_{|F|}(\mathbb{C})$. Taking the direct
limit gives $I_H \cong s_u C^*(\Lambda) s_u \otimes \Kk(\ell^2(X))$.}, we have $sr(I_H) = 2$ as
discussed in Section~\ref{sec.1.3}.

In all three cases, the quotient of $C^*(\Lambda)$ by $I_H$ is $C^*(\Lambda)/I_H\cong
C^*(\Lambda\setminus \Lambda H)\cong C(\TT^2)$. Hence, by \cite[V.3.1.21]{MR2188261} we deduce
that $sr(C^*(\Lambda))\in \{2,3\}$, but we have been unable to determine the exact value in any of
these cases.

Perhaps the easiest-looking case is the $2$-graph (top left) with one red (dashed) edge from $u$
to $v$. In this case $C^*(\Lambda)\cong \mathcal{T}\otimes C(\TT)$, where $\Tt$  is the Toeplitz
algebra generated by the unilateral shift. Despite knowing the stable rank of each of the
components ($sr(\Tt)=2$ and $sr(C(\TT))=1$) the stable rank of the tensor product is not known
(there is no general formula for stable rank of tensor products).

\end{remark}

\end{document}